\documentclass[11pt, a4paper,leqno]{amsart}
\usepackage{amsmath,amsthm,amscd,amssymb,amsfonts, amsbsy}
\usepackage{latexsym}
\usepackage{exscale}
\usepackage[shortlabels]{enumitem}
\usepackage{resmes}
\usepackage{esint}
\usepackage{graphicx}

\usepackage[colorlinks,citecolor=red,pagebackref,hypertexnames=false]{hyperref}
\usepackage{color}

\calclayout
\allowdisplaybreaks

\theoremstyle{plain}
\newtheorem{theorem}[equation]{Theorem}
\newtheorem{lemma}[equation]{Lemma}
\newtheorem{corollary}[equation]{Corollary}

\theoremstyle{definition}
\newtheorem{definition}[equation]{Definition}

\newtheorem{example}[equation]{Example}

\theoremstyle{remark}
\newtheorem{remark}[equation]{Remark}

\numberwithin{equation}{section}

\newcommand{\N}{\mathbb{N}}

\newcommand{\R}{\mathbb{R}}

\newcommand{\Lipc}{\mathrm{Lip}_c}
\newcommand{\e}{\varepsilon}

\newcommand{\BVloc}{BV_{\mathrm{loc}}}
\newcommand{\Df}{\lVert Df\rVert}
\newcommand{\Dfk}{\lVert Df_k\rVert}
\newcommand{\dE}{\lVert \partial E\rVert}
\newcommand{\Dfw}{\lVert Df\rVert_w}
\newcommand{\Dfkw}{\lVert Df_k\rVert_w}
\newcommand{\Dfew}{\lVert Df_\varepsilon\rVert_w}
\newcommand{\dEw}{\lVert \partial E\rVert_w}
\newcommand{\mres}{\mathbin{\vrule height 1.6ex depth 0pt width
0.13ex\vrule height 0.13ex depth 0pt width 1.3ex}}

\usepackage{textcomp}

\DeclareMathOperator*{\essinf}{ess\,inf}

\def\div{\mathop{\operatorname{div}}\nolimits}

\makeatletter
\def\l@subsection{\@tocline{2}{0pt}{2.5pc}{5pc}{}}
\makeatother % To indent subsections in table of contents: https://tex.stackexchange.com/questions/177290/how-to-indent-subsection-in-ams-toc

\newcommand{\joseph}[1]{\marginpar{\scriptsize \textbf{J:} #1}}

\begin{document}
\allowdisplaybreaks

\title[Weighted BV] {Weighted Bounded Variation Revisited}
\author{Simon Bortz}
\address{Simon Bortz\\
         Department of Mathematics\\
         University of Alabama\\
         Tuscaloosa, AL 35487, USA}
\email{sbortz@ua.edu}

\author{Matthew Gossett}
\address{Matthew Gossett\\
         Department of Mathematics\\
         University of Alabama\\
         Tuscaloosa, AL 35487, USA}
\email{mlgossett@crimson.ua.edu}
\author{Joseph Kasel}
\address{Joseph Kasel\\
         Department of Mathematics\\
         University of Alabama\\
         Tuscaloosa, AL 35487, USA}
\email{jekasel@crimson.ua.edu}

\author{Kabe Moen}
\address{Kabe Moen\\
         Department of Mathematics\\
         University of Alabama\\
         Tuscaloosa, AL 35487, USA}
\email{kabe.moen@ua.edu}

\begin{abstract}
In this article, we investigate the theory of weighted functions of bounded variation (BV), as introduced by Baldi \cite{B}. Depending on the theorem, we impose lower semicontinuity and/or a pointwise $A_1$ condition on the weight. Our motivation is twofold: to establish weighted Gagliardo-Nirenberg-Sobolev (GNS) inequalities for BV functions, and to clarify and extend earlier results on weighted BV spaces. 

Our main contributions include a structure theorem under minimal assumptions (lower semicontinuity), a smooth approximation result, an embedding theorem, a weighted GNS inequality for BV functions, and a corresponding weighted isoperimetric inequality.
\end{abstract}

\thanks{Some of this work was completed during a local REU at the University of Alabama in Summer 2025. The authors would like to thank the UA Mathematics Department for financially supporting the REU. S.B. was supported by the Simons Foundation’s Travel Support for Mathematicians program and AWI-CONSERVE. M.G. and J.K. would like to thank UA for supporting them with an undergraduate research stipend.}

\maketitle
\tableofcontents

\section{Introduction}

The purpose of this article is to investigate spaces of functions of bounded variation under a change of measure. Recall that, roughly speaking, the space of bounded variation consists of functions whose distributional derivatives are Radon measures. Compared with Sobolev spaces, $BV$ spaces offer a more flexible framework, as they accommodate functions of a more singular nature—for instance, $f = \chi_E$ when $E$ has finite perimeter. $BV$ spaces have broad applications: they provide generalized solutions to certain PDEs and play a central role in the theory of surface measure and isoperimetric inequalities (see \cite{AFP},\cite{EG},\cite{G84}). The theory of $BV$ functions also plays a fundamental role in total variation denoising and in the Mumford–Shah functional, both of which are instrumental in various aspects of image processing and segmentation. For further applications, we refer the reader to \cite{HV75}.

In this work, we study the weighted space $BV(w)$ associated with a weight $w$, which arises naturally as an extension of the weighted Sobolev space $W^{1,1}(w)$. Weighted $BV$ spaces have been considered by several authors; in particular, we emphasize the contributions of Baldi \cite{B} and Camfield \cite{C}. While \cite{B} is a well-cited reference, our aim is to refine and extend the existing theory, filling in gaps to provide a more complete framework. Specifically, we present a systematic treatment of sets of finite $w$-perimeter, establish density theorems, and apply these results to GNS and isoperimetric inequalities. Our structure theorems differ in important respects from those of Baldi, and we place particular emphasis on the role of the weight by distinguishing between the case where w is merely lower semicontinuous and the stronger setting in which w is both lower semicontinuous and belongs to $A_1$, a class we denote by $A_1^*$; see Section \ref{A1weights} for precise definitions.

{On the other hand, Camfield \cite{C} provides many valuable contributions to the study of weighted $BV$ functions. For example, Camfield proves a characterization of functions of bounded weighted variation (see \cite[Theorem 2.1.5]{C}). We obtain this result using a different approach, but slightly expand on Camfield's result by providing characterizations of functions of locally bounded weighted variation and of functions of bounded weighted variation when the weight is bounded away from zero. Moreover, Camfield proves a smooth approximation result for Lipschitz weights (see \cite[Theorem 4.1.6]{C}). However, we prove a much more general smooth approximation result (see Theorem \ref{thm:EG5.3analogue} and Remark \ref{rmk:wapprox}) that subsumes Camfield's result.}

{Our main contributions also include several significant improvements to the existing literature, in particular \cite{B}. First, we revisit the claimed generality of the results in \cite{B}, which are stated for $A_1^*$ weights. In particular, the approximation argument in \cite{B} is not fully transparent: it is unclear whether the natural Lipschitz approximations of lower semicontinuous functions preserve membership in $A_1^*$, whereas Baldi works initially with weights $w \in \Lipc \cap A_1^*$ and then extends to general $A_1^*$ weights. 

We establish a global Gagliardo--Nirenberg--Sobolev inequality for weighted $BV$ functions and arbitrary $A_1^*$ weights (see Theorem \ref{thm:GNSforweightedBV}), without the additional assumptions imposed in \cite{B}. In contrast, \cite[Theorem 4.2(ii)]{B} also proves a Gagliardo--Nirenberg--Sobolev inequality, but under stronger and somewhat unnatural restrictions on the weight (see Remark \ref{Baldifalse}).

We further obtain a global isoperimetric inequality for sets of finite weighted perimeter under arbitrary $A_1^*$ weights (see Corollary \ref{cor:globalweightedisoperimetric}). Finally, we show that the space of weighted $BV$ functions admits an isometric embedding into the classical unweighted $BV$ space in one higher dimension.}

\subsection{Main Results}

Our first main result is a structure theorem analogous the unweighted structure theorem \cite[Theorem 5.1]{EG}. Compare \cite[Theorem 3.3]{B}, although Baldi restricts to the case of $A_1^*$ weights while we consider weights that are merely positive and lower semicontinuous.

\begin{theorem}[Structure Theorem for $\BVloc(\Omega;w)$]
\label{thm:structuretheorem}
    Let $w:\R^n\to(0,\infty]$ be lower semicontinuous, $f\in BV_\mathrm{loc}(\Omega;w)$. Then, there exist a Radon measure $\Dfw$ and a $\Dfw$-measurable function $\nu:\Omega\to\R^n$ such that
    \begin{enumerate}[(i)]
        \item $|\nu(x)|=1$ $\Dfw$-a.e., and
        \item for all $\varphi\in\Lipc(\Omega;\R^n)$,
        \[
            \int_\Omega f\div\varphi\,dx=-\int_\Omega(\varphi\cdot\nu)\,\frac{1}{w}\,d\Dfw.
        \]
    \end{enumerate}
    In particular, $d\Dfw=w\,d\Df$.
\end{theorem}

As is the case with any function space, we want to show that a collection of ``nicer'' functions approximates functions in our space. In the case of classical BV functions, smooth functions can be used as approximating functions (see \cite[Theorem 5.3]{EG}). We prove a similar theorem in the case of weighted BV functions, although the presence of the weight can cause problems. As a result, we impose an additional condition, the so-called $w$-approximability condition (see Definition \ref{def:approximable}), to ensure we can obtain the desired convergence.

\begin{theorem}[Approximation by Smooth Functions]
\label{thm:EG5.3analogue}
    Let $w\in A_1^*$, $f\in BV(\Omega;w)$.
    \begin{enumerate}[(i)]
        \item If $f$ is $w$-approximable (see Definition \ref{def:approximable}), then there exists a sequence $\{f_k\}_{k=1}^\infty\subseteq C^\infty(\Omega)\cap BV(\Omega;w)$ such that $f_k\to f$ in $L^1(\Omega;w)$ and
        \[
            \lim_{k\to\infty}\Dfkw(\Omega)=\Dfw(\Omega).
        \]
        \item If $f$ is not $w$-approximable, then there exists a sequence $\{f_k\}_{k=1}^\infty\subseteq C^\infty(\Omega)\cap BV(\Omega;w)$ such that $f_k\to f$ in $L^1(\Omega;w)$ and
        \[
            \Dfw(\Omega)\leq\lim_{k\to\infty}\Dfkw(\Omega)\leq [w]_{A_1}\Dfw(\Omega).
        \]
    \end{enumerate}
\end{theorem}

A key application of smooth function approximation is to generalize results about Sobolev functions to BV functions. To that end, we prove a Gagliardo-Nirenberg-Sobolev inequality for $BV(\R^n;w)$ functions.

\begin{theorem}[Gagliardo-Nirenberg-Sobolev Inequality for $BV(\R^n;w)$]
\label{thm:GNSforweightedBV}
    Let $w\in A_1^*$. Then, for all $f\in BV(\R^n;w)$,
    \begin{equation}\label{ineq:weightedBVGNS}
        \lVert f\rVert_{L^{1^*}(\R^n;w)}\leq C_1[w]_{A_1}^{2/1^*}\lVert Df\rVert_{w^{1/1^*}}(\R^n),
    \end{equation}
    where $C_1$ is the constant from Theorem \ref{thm:GNSforweightedsobolev}. If, in addition, $f$ is $w^{1/1^*}$-approximable, then
    \[
        \lVert f\rVert_{L^{1^*}(\R^n;w)}\leq C_1[w]_{A_1}^{1/1^*}\lVert Df\rVert_{w^{1/1^*}}(\R^n).
    \]
\end{theorem}

\begin{remark}
    Note that because we use smooth approximation in the proof, the constant improves when $f$ is $w^{1/1^*}$-approximable. We also remark that by Lemma \ref{lem:deltaapproximable}, the condition that $f$ is $w^{1/1^*}$-approximable holds in particular when $f$ is $w$-approximable.
\end{remark}

{\begin{remark}\label{Baldifalse} Let us compare our result, Theorem \ref{thm:GNSforweightedBV}, with \cite[Theorem 4.2(ii)]{B}, which establishes a one-weight Gagliardo--Nirenberg--Sobolev inequality of the form
\begin{equation}\label{BaldiGNS}
\lVert f\rVert_{L^{q}(\R^n;w)}\leq C_w\lVert Df\rVert_{w}(\R^n)
\end{equation}
provided that $w\in A_1^*$, satisfies the local growth condition
\begin{equation}\label{eqn:balance}
\frac{w(B(x,r))}{w(B(x,s))}\leq C\Big(\frac{r}{s}\Big)^{q'},
\end{equation}
for $0<s\leq r$, and the global growth condition
\begin{equation}\label{eqn:growth}
\limsup_{R\rightarrow \infty} \frac{R}{w(B(0,R))^{\frac1{q'}}}<\infty.
\end{equation}

From our perspective, this result is artificial for several reasons, and this artificiality stems from the proof technique, which extends a local Poincar\'e inequality to a global one by considering increasingly large balls. First, the scaling in the weight $w$ in \eqref{BaldiGNS} is inconsistent with that of our inequality, even in the case $q=n'$. As a consequence, the dependence of the constant $C_w$ on the weight is obscured and cannot be clearly quantified. In particular, $C_w$ cannot depend only on $[w]_{A_1}$ (see \eqref{eqn:A1constant}), as in Theorem \ref{thm:GNSforweightedBV}, since this quantity is invariant under multiplication by positive constants; that is,
$$
[\lambda w]_{A_1}=[w]_{A_1}, \qquad \lambda>0.
$$
To see that this dependence cannot hold, suppose that $C_w=\phi([w]_{A_1})$ for some function $\phi$ on $[1,\infty)$. Taking $w\equiv \lambda>0$ in \eqref{BaldiGNS}, so that $[w]_{A_1}=[\lambda]_{A_1}=1$, with $f\in C_c^\infty(\R^n)$, yields
$$
\lambda^{\frac1q}\lVert f\rVert_{L^{q}(\R^n)}=\lVert f\rVert_{L^{q}(\R^n;w)}
\leq
C_w \lVert Df\rVert_w(\R^n)=\phi(1)\lambda\lVert Df\rVert(\R^n)
$$
Letting $\lambda\to 0$ shows that this is impossible when $q>1$.  Therefore $C_w$ must depend on the weight in other ways, that are not readily apparent.

Condition \eqref{eqn:balance} and \eqref{eqn:growth} are also restrictive. Inequality \eqref{eqn:balance} follows from a reverse H\"older property on the weight $w$, although one depending specifically on $q$, and is therefore not generally satisfied by arbitrary $A_1^*$ weights. The global condition \eqref{eqn:growth} is even more ad hoc. For example, consider the power weight
$$
w(x)=|x|^{\delta-n},
\qquad 0<\delta<1.
$$
It is well known that $w\in A_1^*$. However, for $R>0$,
$$
w(B(0,R))
=
\frac{s_{n-1}}{\delta}R^\delta,
$$
where $s_{n-1}$ is the surface measure of the unit sphere in $\R^n$.  Consequently
$$
\frac{R}{w(B(0,R))^{\frac1{q'}}}
=
C_{n,\delta} R^{1-\frac{\delta}{q'}}
\to \infty
\qquad \text{as } R\to\infty
$$
for any $q\geq 1$.  Thus, condition \eqref{eqn:growth} fails for simple power weights in $A_1^*$, so Theorem 4.2(ii) in \cite{B} does not apply to such weights. In contrast, our result, Theorem \ref{thm:GNSforweightedBV}, applies to all weights in $A_1^*$.
\end{remark}}

One key result for unweighted sets of finite perimeter is the isoperimetric inequality (see \cite[Theorem 5.11]{EG}), which bounds a set's ``area'' by its ``perimeter.'' Taking $f=\chi_E$ in the Gagliardo-Nirenberg-Sobolev inequality (Theorem \ref{thm:GNSforweightedBV}), it is trivial to obtain the following weighted analogue to the isoperimetric inequality.

\begin{corollary}[Global Weighted Isoperimetric Inequality]
\label{cor:globalweightedisoperimetric}
    Let $w\in A_1^*$, $E$ be a set of finite $w$-perimeter in $\R^n$. Then,
    \[
        (w(E))^{1/1^*}\leq C_1[w]_{A_1}^{2/1^*}\dE_{w^{1/1^*}}(\R^n).
    \]
    If, in addition, $\chi_E$ is $w^{1/1^*}$-approximable, then
    \[
        (w(E))^{1/1^*}\leq C_1[w]_{A_1}^{1/1^*}\dE_{w^{1/1^*}}(\R^n).
    \]
\end{corollary}

One thing we would like is to be able to systematically associate functions in $BV(\Omega;w)$ with functions in some unweighted BV space. A similar result is already known for $W^{1,1}(\Omega;w)$ (see Remark \ref{rmk:AntociRemark}). To that end, we formulate the following theorem, which states that $BV(\Omega;w)$ can be isometrically embedded into an unweighted BV space in one higher dimension.

\begin{theorem}[Isometrically Embedding $BV(\Omega;w)\hookrightarrow BV(\Omega_w)$]
\label{thm:isometry}
    Let $w:\mathbb{R}^n\to(0,\infty]$ be lower semicontinuous and let $\Omega\subseteq\mathbb{R}^n$ be open. Then, $J:BV(\Omega;w)\to BV(\Omega_w)$ is an isometric embedding (see Definition \ref{def:J}). That is, for all $f\in BV(\Omega;w)$,
    \[
        \lVert f\rVert_{L^1(\Omega;w)}=\lVert Jf\rVert _{L^1(\Omega_w)}\qquad\textrm{and}\qquad\Dfw(\Omega)=\lVert D(Jf)\rVert(\Omega_w),
    \]
    and it is clear by the definition that $J$ is injective. 
\end{theorem}

Finally, we remark that a weighted analogue of the coarea formula for BV functions has already been proven for very general weights by Camfield \cite[Theorem 3.1.13]{C}, so we will not prove such a result here. In fact, we cite Camfield's result in Section \ref{sec:embedding} (see Theorem \ref{thm:Coarea}).

\subsection{Outline of the Paper}

\begin{itemize}
    \item In Section \ref{sec:preliminaries}, we define classical and weighted BV spaces along with $A_1$ weights.
    \item In Section \ref{sec:structuretheorem}, we prove Theorem \ref{thm:structuretheorem}. Before doing so, we also characterize weighted BV functions.
    \item In Section \ref{sec:finiteperimeter}, we explore sets of finite $w$-perimeter. We prove that $W^{1,1}(\Omega,w)\subsetneq BV(\Omega;w)$ and $W^{1,1}_\textrm{loc}(\Omega,w)\subsetneq BV_\textrm{loc}(\Omega;w)$. Moreover, we consider several examples of that show that sets of finite perimeter do not necessarily have finite $w$-perimeter, and vice versa.
    \item In Section \ref{sec:smoothapproximation}, we prove Theorem \ref{thm:EG5.3analogue}. We also consider the optimality of the $w$-approximability condition (see Definition \ref{def:approximable}) in obtaining Theorem \ref{thm:EG5.3analogue}(i).
    \item In Section \ref{sec:isoperimetricinequality}, we prove Theorem \ref{thm:GNSforweightedBV}.
    \item In Section \ref{sec:embedding}, we prove Theorem \ref{thm:isometry}.
    \item In Appendix \ref{sec:maximalfunctionappendix}, we characterize the measures that satisfy the hypotheses of Theorem \ref{thm:GNSforweightedsobolev}.
\end{itemize}

%%%%%%%%%%%%%%%%%%%%%%%%%%%%%%%%%%%%%%%%%%%%%%%%
%%%%%%%%%%%%%%%%%%%%%%%%%%%%%%%%%%%%%%%%%%%%%%%%
%%%%%%%%%%%%%%%%%%%%%%%%%%%%%%%%%%%%%%%%%%%%%%%%
%%%%%%%%%%%%%%%%%%%%%%%%%%%%%%%%%%%%%%%%%%%%%%%%
%%%%%%%%%%%%%%%%%%%%%%%%%%%%%%%%%%%%%%%%%%%%%%%%
%%%%%%%%%%%%%%%%%%%%%%%%%%%%%%%%%%%%%%%%%%%%%%%%

%\newpage

\section{Preliminaries}
\label{sec:preliminaries}

\subsection{Notation}

We will use the following notation:
\begin{itemize}
    \item Throughout the paper, we let $n\in\N$, and we use $\Omega$ to denote an open subset of $\R^n$.
    \item We use the letters $c$, $C$ to denote harmless positive constants, not necessarily the same at each occurrence, which depend only on dimension and the constants appearing in the hypotheses of the theorems (which we refer to as the ``allowable parameters''). We shall also sometimes write $a\lesssim b$ and $a\approx b$ to mean, respectively, that $a\leq Cb$ and $0<c\leq a/b\leq C$, where the constants $c$ and $C$ are as above, unless explicitly noted to the contrary.
\end{itemize}

\subsection{Classical $BV$ Spaces}
Following \cite{EG}, we recall the definitions of functions of bounded variation and sets of finite perimeter.

\begin{definition}[{\cite[Definitions 5.1 and 5.2]{EG}}]
\,
    \begin{enumerate}[(i)]
        \item Let $f\in L^1(\Omega)$. Then, we say that $f$ has \textbf{bounded variation} in $\Omega$ if
        \[
            \sup\left\{\int_\Omega f\div\varphi\,d x:\varphi\in\Lipc(\Omega;\R^n),|\varphi|\leq1\right\}<\infty.
        \]
        We denote the space of such functions by $BV(\Omega)$.
        \item Let $f\in L^1_\text{loc}(\Omega)$. Then, we say that $f$ has \textbf{locally bounded variation} in $\Omega$ if
        \[
            \sup\left\{\int_V f\div\varphi\,d x:\varphi\in\Lipc(V;\R^n),|\varphi|\leq1\right\}<\infty
        \]
        for all $V\Subset\Omega$. We denote the space of such functions by $\BVloc(\Omega)$.
        \item We say that a set $E$ has \textbf{finite perimeter} (resp. \textbf{locally finite perimeter}) in $\Omega$ if $\chi_E\in BV(\Omega)$ (resp. $\chi_E\in\BVloc(\Omega)$).
    \end{enumerate}
\end{definition}

We remark that we will identify functions of bounded variation that agree a.e. In the definition given in \cite{EG}, the spaces are introduced with respect to the test space $C_c^1(\Omega;\R^n)$ rather than $\mathrm{Lip}_c(\Omega;\R^n)$. This distinction poses no difficulty, however, since the entire framework extends naturally to Lipschitz test functions (see \cite{Fed69}).

Now, we recall the structure theorem for functions of locally bounded variation.

\begin{theorem}[{\cite[Theorem 5.1]{EG}}, Structure Theorem for $BV_\mathrm{loc}(\Omega)$]
\label{thm:EG5.1}
    Let $f\in BV_\mathrm{loc}(\Omega)$. Then, there exist a Radon measure $\mu$ on $\Omega$ and a $\mu$-measurable function $\nu:\Omega\to\R^n$ such that
    \begin{enumerate}[(i)]
        \item $|\nu(x)|=1$ $\mu$-a.e., and
        \item for all $\varphi\in\Lipc(\Omega;\R^n)$, we have
        \[
            \int_\Omega f\div\varphi\,dx=-\int_\Omega\varphi\cdot\nu\,d\mu.
        \]
    \end{enumerate}
\end{theorem}

We recall the notation from \cite{EG}. Namely, we write
\[
    \Df:=\mu,\qquad\text{and}\qquad[Df]:=\Df\resmes\nu,
\]
where $\mu$ and $\nu$ are as in Theorem \ref{thm:EG5.1}. In particular, if $f=\chi_E$, then we write
\[
    \dE:=\mu,\qquad\text{and}\qquad\nu_E:=-\nu.
\]
And if $f\in W^{1,1}(\Omega)$, then
\[
    \Df=\mathcal{L}^n\resmes|Df|,
\]
where $\mathcal{L}^n$ is the $n$-dimensional Lebesgue measure, and $Df$ is the weak gradient of $f$.

Finally, note that for each open set $V\Subset\Omega$,
\[
    \Df(V)=\sup\left\{\int_V f\div\varphi\,dx:\varphi\in\Lipc(V;\R^n),|\varphi|\leq1\right\},
\]
and
\[
    \dE(V)=\sup\left\{\int_E \div\varphi\,dx:\varphi\in\Lipc(V;\R^n),|\varphi|\leq1\right\}.
\]

\subsection{Weighted $BV$ Spaces}

Following \cite{B}, we define functions of bounded weighted variation and sets of finite weighted perimeter.

\begin{definition}
{Let $w:\R^n\to(0,\infty]$.}%\joseph{\Rd What is correct here: $[0,\infty]$ or $(0,\infty]$? For our results requiring $w\in A_1^*$, we have $w>0$ since the only everywhere $A_1$ weight that is not strictly positive is $w\equiv0$. However, for our results only needing lower semicontinuity, can we allow the weight to be $0$?}
    \begin{enumerate}[(i)]
        \item Let $f\in L^1(\Omega;w)$. Then, we say that $f$ has \textbf{bounded $w$-variation} if
        \[
            \Dfw(\Omega):=\sup\left\{\int_\Omega f\div\varphi\,dx:\varphi\in\Lipc(\Omega;\R^n),|\varphi|\leq w\right\}<\infty.
        \]
        We denote the space of such functions by $BV(\Omega;w)$.
        \item Let $f\in L^1_\textrm{loc}(\Omega;w)$. Then, we say that $f$ has \textbf{locally bounded $w$-variation} if
        \[
            \Dfw(V):=\sup\left\{\int_V f\div\varphi\,dx:\varphi\in\Lipc(V;\R^n),|\varphi|\leq w\right\}<\infty
        \]
        for all $V\Subset\Omega$. We denote the space of such functions by $\BVloc(\Omega;w)$.
        \item We say that a set $E$ has \textbf{finite $w$-perimeter} (resp. \textbf{locally finite $w$-perimeter}) in $\Omega$ if $\chi_E\in BV(\Omega;w)$ (resp. $\chi_E\in\BVloc(\Omega;w)$).
    \end{enumerate}
\end{definition}
As in the unweighted case, we will identify functions of bounded variation that agree a.e. 

Now, we record the following fact relating weighted and unweighted $BV$ spaces.

\begin{lemma}[Relationship between Weighted and Unweighted $BV$ Spaces]
\label{lem:BVwcontainment}
Let $w:\R^n\to(0,\infty]$ be lower semicontinuous.
\begin{enumerate}[(i)]
    \item $BV(\Omega;w)\subseteq BV_\mathrm{loc}(\Omega;w)\subseteq BV_\mathrm{loc}(\Omega)$.
    \item If $w\geq c>0$ on $\Omega$, then $BV(\Omega;w)\subseteq BV(\Omega)$.
\end{enumerate}
\end{lemma}

\begin{remark}
    The assumption that $w\geq c>0$ in Lemma \ref{lem:BVwcontainment}(ii) holds trivially if $\Omega$ is bounded.
\end{remark}

\begin{proof}
    The first containment of (i) is trivial. Then, for all open $V\Subset\Omega$,
    \begin{align*}
        &\sup\left\{\int f\div\varphi\,dx:\varphi\in\Lipc(V;\R^n),|\varphi|\leq 1\right\}\\
        &\qquad\leq\sup\left\{\int f\div\varphi\,dx:\varphi\in\Lipc(V;\R^n),|\varphi|\leq\frac{w}{\inf_V w}\right\}\\
        &\qquad\leq\frac{1}{\inf_V w}\sup\left\{\int f\div\varphi\,dx:\varphi\in\Lipc(V;\R^n),|\varphi|\leq w\right\}\\
        &\qquad\leq\frac{\Dfw(V)}{\inf_V w}\\
        &\qquad<\infty,
    \end{align*}
    where we used that $w$ is bounded away from 0 on the bounded set $V$ and $f\in BV_\textrm{loc}(\Omega;w)$. This gives the second containment of (i).

    (ii) holds by repeating the argument above, replacing $V$ with $\Omega$.
\end{proof}

\subsection{$A_1$ Weights}\label{A1weights}

We will now define the class of weights that will be of particular interest to us.

\begin{definition}
    Let $w:\R^n\to[0,\infty]$. We say that $w$ is an \textbf{$A_1$ weight} if $w\in L^1_\text{loc}(\Omega)$, and there exists some $C>0$ such that
    \begin{equation}
    \label{eqn:A1condition}
        \fint_B w\,dx\leq C\essinf_{x\in B}w(x)
    \end{equation}
    for all balls $B\subseteq\R^n$. In this case, we write $w\in A_1$. We call the smallest $C$ for which (\ref{eqn:A1condition}) holds the \textbf{$A_1$ constant} and write
    \begin{equation}
    \label{eqn:A1constant}
        [w]_{A_1}:=\inf\{C:\text{(\ref{eqn:A1condition}) holds}\}.
    \end{equation}
    If, in addition, $w$ is lower semicontinuous, we say that $w$ is an \textbf{$A_1^*$ weight} and write $w\in A_1^*$.
\end{definition}

In particular, note that condition (\ref{eqn:A1condition}) immediately implies that
\[
    Mw(x)\leq[w]_{A_1}w(x)\qquad\textrm{for all }w\in A_1,\textrm{ and a.e. }x\in\R^n,
\]
where $M$ is the Hardy-Littlewood maximal function taken over uncentered balls. This fact will become quite important in several proofs of ours. However, because functions can have positive variation on sets of measure zero, it is not enough to have this inequality a.e. Thus, we define the following slightly stronger subclass of $A_1$ weights.

\begin{definition}
    Let $w:\R^n\to[0,\infty]$. We say that $w$ is an \textbf{everywhere $A_1$ weight} if $w\in L^1_\text{loc}(\Omega)$, and there exists some $C>0$ such that
    \begin{equation}
    \label{eqn:EA1condition}
        \fint_B w\,dx\leq C\inf_{x\in B}w(x)
    \end{equation}
    for all balls $B\subseteq\R^n$. In this case, we write $w\in A_1$. We call the smallest $C$ for which (\ref{eqn:EA1condition}) holds the \textbf{$A_1$ constant} and write
    \[
        [w]_{A_1}:=\inf\{C:\text{(\ref{eqn:EA1condition}) holds}\}.
    \]
    If, in addition, $w$ is lower semicontinuous, we say that $w$ is an \textbf{everywhere $A_1^*$ weight} and write $w\in A_1^*$.
\end{definition}

\begin{remark}
    By abuse of notation, we will denote the collections of everywhere $A_1$ weights and everywhere $A_1^*$ weights as $A_1$ and $A_1^*$, respectively. Thus, in the sequel, we mean by $w\in A_1$ or $w\in A_1^*$ that $w$ is an everywhere $A_1$ weight or an everywhere $A_1^*$ weight, respectively.
\end{remark}

Because the essential infimum is replaced by an infimum in condition (\ref{eqn:EA1condition}), we get that
\begin{equation}
\label{eqn:maxmlfxnbound}
    Mw(x)\leq[w]_{A_1}w(x)\qquad\textrm{for all }w\in A_1,x\in\R^n.
\end{equation}
Note also that $w\in A_1$ implies that $w\equiv 0$ or $w>0$ everywhere. We will exclude the trivial case that $w\equiv0$ and assume that $w\in A_1$ implies that $w$ is positive. The following estimate will be of particular use to us.  The classical proof can be found in \cite[Theorem 2.1.10]{Graf}.

\begin{lemma}
\label{lem:mollificationbound}
    Let $w\in A_1^*$, $\eta \in C_c^\infty(\R^n)$ be a positive radially decreasing function with $\int_{\mathbb R^n}\eta \,dx=1$. Then, for any $\e>0$
    \[
        \eta_\e*w(x)\leq [w]_{A_1}w(x).
    \]
 %   \joseph{I still don't know how to get a better estimate than this. I know that Kabe keeps writing $\eta_e*w\leq[w]_{A_1}w$, but I can't seem to derive that.}
\end{lemma}

\begin{proof} Since $\eta$ is a positive radially decreasing function with integral one, we have
$$\eta_\e*w(x)\leq Mw(x)\leq [w]_{A_1}w(x).$$
  %  Using (\ref{eqn:maxmlfxnbound}),
  %  \begin{align*}
  %      \eta_\e*w(x)&=\int_{B(x,\e)}\frac{1}{\e^n}\eta\left(\frac{x-y}{\e}\right)w(y)\,dy\\
  %      &\leq\frac{\lVert\eta\rVert_\infty|B(0,1)|}{\e^n|B(0,1)|}\int_{B(x,\e)}w(y)\,dy\\
  %      &\leq\lVert\eta\rVert_\infty|B(0,1)|Mw(x)\\
 %      &\leq\lVert\eta\rVert_\infty|B(0,1)|[w]_{A_1}w(x).
 %   \end{align*}
\end{proof}

%%%%%%%%%%%%%%%%%%%%%%%%%%%%%%%%%%%%%%%%%%%%%%%%
%%%%%%%%%%%%%%%%%%%%%%%%%%%%%%%%%%%%%%%%%%%%%%%%
%%%%%%%%%%%%%%%%%%%%%%%%%%%%%%%%%%%%%%%%%%%%%%%%
%%%%%%%%%%%%%%%%%%%%%%%%%%%%%%%%%%%%%%%%%%%%%%%%
%%%%%%%%%%%%%%%%%%%%%%%%%%%%%%%%%%%%%%%%%%%%%%%%
%%%%%%%%%%%%%%%%%%%%%%%%%%%%%%%%%%%%%%%%%%%%%%%%

%\newpage

\section{A Structure Theorem for $\BVloc(\Omega;w)$}
\label{sec:structuretheorem}

Before proving a structure theorem from $BV(\Omega;w)$, we will prove a theorem regarding the relationship between the weighted and unweighted variation measures similar to \cite[Theorem 4.1]{B}. We remark, however, that Baldi's theorem assumes that the weights under consideration are $A_1^*$ weights, while our result considers weights that are merely positive and lower semicontinuous. As a result, our proof differs significantly from Baldi's.%\joseph{\Rd Maybe also mention Camfield's result here?}

\begin{theorem}[Relationship between Weighted and Unweighted Variation Measure]
\label{thm:varwmeasure}
Let $w:\R^n\to(0,\infty]$ be lower semicontinuous.
\begin{enumerate}[(i)]
    \item $f\in BV(\Omega;w)$ if and only if $f\in BV_\mathrm{loc}(\Omega)$ and $w\in L^1(\Omega;d\Df)$. In this case,
    \begin{equation}
    \label{eqn:varwiswdmu}
        \Dfw(\Omega)=\int_\Omega w\,d\Df.
    \end{equation}
    \item $f\in\BVloc(\Omega;w)$ if and only if $f\in\BVloc(\Omega)$ and $w\in L^1_\mathrm{loc}(\Omega;d\Df)$. In this case,
    \[
        \Dfw(V)=\int_V w\,d\Df
    \]
    for all $V\Subset\Omega$.
    \item Suppose $w\geq c>0$ on $\Omega$. Then, $f\in BV(\Omega;w)$ if and only if $f\in BV(\Omega)$ and $w\in L^1(\Omega;d\Df)$. In this case,
    \[
        \Dfw(\Omega)=\int_\Omega w\,d\Df.
    \]
\end{enumerate}
\end{theorem}

\begin{remark}
    We remark here that the condition that $w\geq c>0$ holds trivially if $\Omega$ is bounded.
\end{remark}

\begin{proof}
    We will first prove the forward direction of (i). To that end, suppose $f\in BV(\Omega;w)$. By Lemma \ref{lem:BVwcontainment}(i), $f\in BV_\textrm{loc}(\Omega)$. Then, by Theorem \ref{thm:EG5.1}, there exists a $\Df$-measurable function $\nu:\Omega\to\R^n$ such that
    \begin{equation}
    \label{eqn:EGnu=1ae}
        |\nu(x)|=1\qquad\Df\textrm{-a.e}.
    \end{equation}
    and
    \begin{equation}
    \label{eqn:EGstructthmeqn}
        \int_\Omega f\div\varphi\,dx=-\int_\Omega \varphi\cdot\nu\,d\Df
    \end{equation}
    for all $\varphi\in\Lipc(\Omega;\R^n)$. By (\ref{eqn:EGstructthmeqn}) and the definition of $\Dfw(\Omega)$, we get that
    \begin{equation}
    \label{eqn:varwisphidotnu}
        \Dfw(\Omega)=\sup\left\{\int_\Omega \varphi\cdot\nu\,d\Df:\varphi\in\Lipc(\Omega;\R^n),|\varphi|\leq w\right\}
    \end{equation}
    Now, note that if $|\varphi|\leq w$, then $|\varphi\cdot\nu|\leq w|\nu|\leq w$ $\Df$-a.e. By this fact and (\ref{eqn:varwisphidotnu}), we have that
    \[
        \Dfw(\Omega)\leq\int_\Omega w\,d\Df.
    \]
    It remains to show the inequality in the other direction.

    To that end, we first fix an open set $V\Subset\Omega$ and let $\delta>0$. Since $V\Subset\Omega$ and $f\in\BVloc(\Omega)$, note that $\Df(V)<\infty$. Next, we define a new function $\nu':\Omega\to\R^n$ by
    \[
        \nu'(x)=\begin{cases}
            \nu(x)&\textrm{if }|\nu(x)|=1\\
            0&\textrm{otherwise}.
        \end{cases}
    \]
    By (\ref{eqn:EGnu=1ae}), $\nu'=\nu$ $\Df$-a.e. By definition, $|\nu'(x)|\leq 1$ for all $x\in\Omega$. Thus, we may invoke \cite[Theorem 1.15]{EG} to obtain a continuous function $\overline{\nu}_\delta:\R^n\to\R^n$ so that
    \[
        \mu(\{x\in V:\overline{\nu}_\delta(x)\neq\nu'(x)\})<\delta.
    \]
    In addition, the construction in \cite{EG} ensures that $|\overline{\nu}_\delta(x)|\leq\sup_\Omega |\nu'(x)|\leq 1$. Now, let $\eta_\e$ be the standard mollifier, and set $\overline{\nu}_{\e,\delta}=\overline{\nu}_\delta*\eta_\e$. Then, $\overline{\nu}_{\e,\delta}\to\overline{\nu}_\delta$ on $\R^n$ and $\overline{\nu}_{\e,\delta}\in C^\infty(\R^n)$ for all $\e>0$. Thus, for any nonnegative $u\in\Lipc(V)$ with $u\leq w$ and $\delta>0$, $u\overline{\nu}_{\e,\delta}\in\Lipc(V;\R^n)$ with $|u\overline{\nu}_{\e,\delta}|\leq w$. Thus,
    \begin{align*}
        \Dfw(\Omega)&=\sup\left\{\int_\Omega \varphi\cdot\nu\,d\Df:\varphi\in\Lipc(\Omega;\R^n),|\varphi|\leq w\right\}\\
        &\geq\lim_{\e\to0^+}\int_V u\overline{\nu}_{\e,\delta}\cdot\nu\,d\Df\\
        &=\int_V u\overline{\nu}_\delta\cdot\nu\,d\Df\\
        &=\int_{V\cap\{\overline{\nu}_\delta=\nu'\}}u\,d\Df+\int_{V\cap\{\overline{\nu}_\delta\neq\nu'\}}u\overline{\nu}_\delta\cdot\nu\,d\Df,
    \end{align*}
    where in the second to last equality, we used the Dominated Convergence Theorem, and in the last equality, we used the fact that $\nu'=\nu$ $\mu$-a.e. Taking $\delta\to 0^+$ and applying the Dominated Convergence Theorem again, we obtain
    \[
        \Dfw(\Omega)\geq\int_V u\,d\Df
    \]
    for all nonnegative $u\in\Lipc(V)$ with $u\leq w$. In particular, if we choose a nonnegative, increasing sequence $\{w_k\}_{k=1}^\infty\subseteq\Lipc(V)$ such that $w_k\to w$, then
    \[
        \Dfw(\Omega)\geq\lim_{k\to\infty}\int_V w_k\,d\Df=\int_V w\,d\Df
    \]
    by the Monotone Convergence Theorem. Finally, we note that $V\Subset\Omega$ was arbitrary. Thus, we can choose an ascending sequence of open sets $V_m\Subset\Omega$ such that $\bigcup_{m=1}^\infty V_m=\Omega$ and use the Monotone Convergence Theorem to get
    \[
        \Dfw(\Omega)\geq\lim_{m\to\infty}\int_{V_m}w\,d\Df=\int_\Omega w\,d\Df.
    \]
    This shows the inequality in the other direction. Finally, the equality
    \[
        \Dfw(\Omega)=\int_\Omega w\,d\Df
    \]
    immediately gives that $w\in L^1(\Omega;d\Df)$ since $f\in BV(\Omega;w)$. This shows the forward direction, and additionally shows (\ref{eqn:varwiswdmu}).

    For the backward direction of (i), suppose $f\in BV_\mathrm{loc}(\Omega)$ and $w\in L^1(\Omega;d\Df)$. By Theorem \ref{thm:EG5.1}, there exists a $\Df$-measurable function $\nu:\Omega\to\R^n$ such that $|\nu(x)|=1$ $\Df$-a.e. and
    \[
        \int_\Omega f\div\varphi\,dx=-\int_\Omega \varphi\cdot\nu\,d\Df
    \]
    for all $\varphi\in\Lipc(\Omega;\R^n)$. For all $\varphi\in\Lipc(\Omega;\R^n)$ with $|\varphi|\leq w$, $|\varphi\cdot\nu|\leq w$ $\Df$-a.e. Hence, for all $\varphi\in\Lipc(\Omega;\R^n)$ with $|\varphi|\leq w$,
    \[
        \int_\Omega f\div\varphi\,dx=-\int_\Omega \varphi\cdot\nu\,d\Df\leq\int_\Omega w\,d\Df<\infty,
    \]
    where we used that $w\in L^1(\Omega;d\Df)$. Thus,
    \[
        \Dfw(\Omega)=\sup\left\{\int_\Omega f\div\varphi\,dx:\varphi\in\Lipc(\Omega;\R^n),|\varphi|\leq w\right\}\leq\int_\Omega w\,d\Df<\infty,
    \]
    so $f\in BV(\Omega;w)$. This shows the backwards direction of (i).

    The proof of (ii) is analogous to the proof (i) by simply replacing $\Omega$ by $V\Subset\Omega$ when necessary. And (iii) follows from (i) and Lemma \ref{lem:BVwcontainment}(ii).
\end{proof}

With Theorem \ref{thm:varwmeasure} in hand, the proof of Theorem \ref{thm:structuretheorem} is easy.

\begin{proof}[Proof of Theorem \ref{thm:structuretheorem}]
    This proof follows from by substituting $d\Dfw=w\,d\Df$ into the unweighted structure theorem \cite[Theorem 5.1]{EG}.
\end{proof}

%%%%%%%%%%%%%%%%%%%%%%%%%%%%%%%%%%%%%%%%%%%%%%%%
%%%%%%%%%%%%%%%%%%%%%%%%%%%%%%%%%%%%%%%%%%%%%%%%
%%%%%%%%%%%%%%%%%%%%%%%%%%%%%%%%%%%%%%%%%%%%%%%%
%%%%%%%%%%%%%%%%%%%%%%%%%%%%%%%%%%%%%%%%%%%%%%%%
%%%%%%%%%%%%%%%%%%%%%%%%%%%%%%%%%%%%%%%%%%%%%%%%
%%%%%%%%%%%%%%%%%%%%%%%%%%%%%%%%%%%%%%%%%%%%%%%%

%\newpage

\section{Sets of Finite $w$-Perimeter}
\label{sec:finiteperimeter}

A natural question to ask is whether every positive, lower semicontinuous weight $w$ admits a set of finite $w$-perimeter. The following lemma answers this question affirmatively. Namely, in the unweighted setting, we have that $W^{1,1}(\Omega)\subsetneq BV(\Omega)$ and $W^{1,1}_{\textrm{loc}}(\Omega)\subsetneq\BVloc(\Omega)$, where the fact that the containments are proper is shown by the existence of sets of finite perimeter. See, for example, \cite[pp. 197-198]{EG}. We now prove the equivalent statement in the weighted setting.

\begin{lemma}
\label{lem:sobolevcontainment}
    Let $w:\R^n\to(0,\infty]$ be lower semicontinuous. Then, $W^{1,1}(\Omega;w)\subsetneq BV(\Omega;w)$, and $W^{1,1}_\mathrm{loc}(\Omega;w)\subsetneq\BVloc(\Omega;w)$.
\end{lemma}

\begin{proof}
    The proof of each containment is essentially the same, so we will only prove the first containment.

    To that end, suppose $f\in W^{1,1}(\Omega;w)$. Then, for all $\varphi\in\Lipc(\Omega;\R^n)$ with $|\varphi|\leq w$, integration by parts yields
    \[
        \int_\Omega f\div\varphi\,dx=-\int_\Omega Df\cdot\varphi\,dx\leq\int_\Omega |Df|\,w\,dx=\lVert Df\rVert_{L^1(\Omega;w)}<\infty.
    \]
    Thus,
    \[
        \Dfw(\Omega)\leq\lVert Df\rVert_{L^1(\Omega;w)}<\infty,
    \]
    so $f\in BV(\Omega;w)$.

    Next, we must show that the containment is proper. To that end, first note that (after translating $\Omega$ if necessary) there exists some $\e>0$ such that $B(0,\e)\subseteq\Omega$. Then, a change of variables to polar coordinates yields that
    \begin{equation}
    \label{eqn:polarcoords}
        \int_{B(0,\e)}w(x)\,dx=\int_0^\e r^{n-1}\int_{|\theta|=1}w(r,\theta)\,d\mathcal{H}^{n-1}(\theta)\,dr.
    \end{equation}
    Note that the left-hand side is finite since $w$ is locally integrable. Now, suppose for the sake of obtaining a contradiction that $\chi_{B(0,\delta)}\not\in BV(\Omega;w)$ for all $0<\delta<\e$. Then, for all $0<\delta<\e$,
    \[
        \int_{|\theta|=1}w(\delta,\theta)\,d\mathcal{H}^{n-1}(\theta)=\int_{\partial B(0,\delta)}w\,d\mathcal{H}^{n-1}=\int_{\partial B(0,\delta)}w\,d\lVert\partial B(0,\delta)\rVert=\infty,
    \]
    where in the last equality we used Theorem \ref{thm:varwmeasure}(i). This implies that the right-hand side of (\ref{eqn:polarcoords}) is infinite, a contradiction. Thus, there exists some $0<\delta<\e$ such that $\chi_{B(0,\delta)}\in BV(\Omega;w)$. It is certainly the case that $\chi_{B(0,\delta)}\not\in W^{1,1}(\Omega;w)$, so this shows that the containment is proper.
\end{proof}

\begin{remark}
    These containments are important, as they ensure that there exists a set of finite $w$-perimeter, no matter the weight $w$. In fact, the proof above shows that if $B(x,R)\subseteq\Omega$, then $B(x,r)$ is a set of finite $w$-perimeter for a.e. $r\in(0,R]$.
\end{remark}

\begin{remark}
    \label{rmk:weakderivativeisvariation}
    In fact, if $f\in W^{1,1}(\Omega;w)$, then
    \[
        \Dfw(\Omega)=\lVert Df\rVert_{L^1(\Omega;w)}.
    \]
    Indeed, we have that $W^{1,1}(\Omega;w)\subseteq W^{1,1}_\textrm{loc}(\Omega)$ (by a similar argument to Lemma \ref{lem:BVwcontainment}), so by an example on pages 197-198 of \cite{EG}, we have that $d\Df=|Df|\,dx$. Hence, by Theorem \ref{thm:varwmeasure}(i),
    \[
        \Dfw(\Omega)=\int_\Omega w\,d\Df=\int_\Omega |Df|\,w\,dx=\lVert Df\rVert_{L^1(\Omega;w)}.
    \]
\end{remark}

Now, note that we have from Lemma \ref{lem:BVwcontainment}(i) that $BV(\Omega;w)\subseteq BV_\textrm{loc}(\Omega)$. Thus, every set of finite $w$-perimeter in $\Omega$ has locally finite perimeter in $\Omega$. And by Lemma \ref{lem:BVwcontainment}(ii), if $w\geq c>0$ on $\Omega$, then every set of finite $w$-perimeter in $\Omega$ has finite perimeter in $\Omega$. In general, however, there can exist a set of finite $w$-perimeter in $\Omega$ that does not have finite perimeter in $\Omega$. Conversely, there can exist a set of finite perimeter in $\Omega$ that does not have finite $w$-perimeter in $\Omega$. The following examples illustrate these facts.

\begin{example}
    Consider $\Omega=\R^n$, $n\geq2$,
    \[
        w(x)=\begin{cases}
            |x|^{-n+\frac{1}{2}}&\textrm{if }|x|>1\\
            1&\textrm{if }|x|\leq1,
        \end{cases}
    \]
    and $E=\R^{n-1}\times(-1,1)$. Then, by \cite[Theorem 5.16]{EG}, for all $\varphi\in\Lipc(\R^n;\R^n)$,
    \[
        \int_E \div\varphi\,dx=\int_{\partial E}\varphi\cdot\nu\,d\mathcal{H}^{n-1},
    \]
    where $\nu(x)=(0,\ldots,0,-1)$ for all $x\in\R^{n-1}\times\{-1\}$ and $\nu(x)=(0,\ldots,0,1)$ for all $x\in\R^{n-1}\times\{1\}$. Choosing $\varphi$ that approximate $\nu$, we see that
    \[
        \dE(\R^n)=\int_{\partial E}d\mathcal{H}^{n-1}=\infty,
    \]
    so $E$ does not have finite perimeter in $\R^n$. However, for all $\varphi\in\Lipc(\R^n;\R^n)$ with $|\varphi|\leq w$, we have that $|\varphi\cdot\nu|\leq w$. Thus,
    \[
        \dEw(\R^n)\leq\int_{\partial E}w\,d\mathcal{H}^{n-1}<\infty,
    \]
    so $E$ does have finite $w$-perimeter in $\R^n$.
\end{example}

\iffalse
\begin{example}
    Consider $\Omega=\R$, $w(x)=|x|^{-1/2}$, and\joseph{We might have to remove this example since $\chi_E\not\in L^1(\R;w)$, which a requirement by the way that we have defined $BV(\R;w)$.}
    \[
        E=\bigcup_{n=1}^\infty\left((2n)^4,(2n+1)^4\right).
    \]
    Then, by \cite[Theorem 5.16]{EG}, for all $\varphi\in\Lipc(\R)$,
    \[
        \int_E\div\varphi\,dx=\int_{\partial E}\varphi\nu\,d\mathcal{H}^0,
    \]
    where
    \[
        \nu(x)=\begin{cases}
            -1&\textrm{if }x=(2n)^4\\
            1&\textrm{if }x=(2n+1)^4.
        \end{cases}
    \]
    By the properties of the counting measure $\mathcal{H}^0$,
    \[
        \int_{\partial E}\varphi\nu\,d\mathcal{H}^0=\sum_{x\in\partial E}\varphi(x)\nu(x)=\sum_{m=2}^\infty \varphi(m^4)\nu(m^4).
    \]
    Choosing $\varphi$ that approximate $\nu$, we see that
    \[
        \dE(\R)=\sum_{m=2}^\infty 1=\infty,
    \]
    so $E$ does not have finite perimeter in $\R$. However, for all $\varphi\in\Lipc(\R)$ with $|\varphi|\leq w$, we have that $|\varphi\nu|\leq w$. Thus,
    \[
        \dEw(\R)\leq\sum_{m=2}^\infty w(m^4)=\sum_{m=2}^\infty\frac{1}{m^2}<\infty,
    \]
    so $E$ does have finite $w$-perimeter.
\end{example}
\fi

\begin{example}
    Consider $\Omega=\R$, $w(x)=|x|^{-1/2}$, and $E=(0,1)$. Then, by \cite[Theorem 5.16]{EG}, for all $\varphi\in\Lipc(\R)$,
    \[
        \int_E\div\varphi\,dx=\int_{\partial E}\varphi\nu\,d\mathcal{H}^0=\varphi(1)-\varphi(0),
    \]
    where $\nu(0)=-1$ and $\nu(1)=1$. For $|\varphi|\leq 1$,
    \[
        \int_E\div\varphi\,dx\leq|\varphi(1)-\varphi(0)|\leq|\varphi(1)|+|\varphi(0)|\leq 2.
    \]
    Hence,
    \[
        \dE(\Omega)\leq 2<\infty,
    \]
    so $E$ has finite perimeter in $\R$. However, for $|\varphi|\leq w$, letting $\varphi$ approximate $-w$ gives
    \[
        \dEw(\Omega)\geq w(0)-w(1)=\infty,
    \]
    so $E$ does not have finite $w$-perimeter.
\end{example}

%%%%%%%%%%%%%%%%%%%%%%%%%%%%%%%%%%%%%%%%%%%%%%%%
%%%%%%%%%%%%%%%%%%%%%%%%%%%%%%%%%%%%%%%%%%%%%%%%
%%%%%%%%%%%%%%%%%%%%%%%%%%%%%%%%%%%%%%%%%%%%%%%%
%%%%%%%%%%%%%%%%%%%%%%%%%%%%%%%%%%%%%%%%%%%%%%%%
%%%%%%%%%%%%%%%%%%%%%%%%%%%%%%%%%%%%%%%%%%%%%%%%
%%%%%%%%%%%%%%%%%%%%%%%%%%%%%%%%%%%%%%%%%%%%%%%%

%\newpage

\section{Smooth Approximation in $BV(\Omega;w)$}
\label{sec:smoothapproximation}

Our goal in this section is to prove Theorem \ref{thm:EG5.3analogue}, a weighted analogue to \cite[Theorem 5.3]{EG}, which constructs smooth approximations for functions in $BV(\Omega)$. We begin by proving a weighted analogue for \cite[Theorem 5.2]{EG}.

\begin{theorem}[Lower Semicontinuity of $\Dfw$]
\label{thm:EG5.2analogue}
    Let $w:\R^n\to(0,\infty]$ be lower semicontinuous. Suppose $\{f_k\}_{k=1}^\infty\subseteq BV(\Omega;w)$ and $f_k\to f$ in $L^1_\mathrm{loc}(\Omega;w)$. Then,
    \[
        \Dfw(\Omega)\leq\liminf_{k\to\infty}\Dfkw(\Omega).
    \]
\end{theorem}

%\joseph{I cut out some here. I'm not sure if we should cut out more or not.}
\begin{proof}
    By assumption, for all compact $K\subseteq\Omega$,
    \[
        \lVert f_k-f\rVert_{L^1(K;w)}=\int_K |f_k-f|\,w\,dx\underset{k\to\infty}{\longrightarrow}0.
    \]
    Since $K$ is bounded and $w$ is positive and lower semicontinuous, $w$ is bounded away from 0 on $K$, say $w\geq c>0$ on $K$. Thus,
    \[
        \lVert f_k-f\rVert_{L^1(K)}=\int_K |f_k-f|\,dx\leq\frac{1}{c}\int_K |f_k-f|\,w\,dx\underset{k\to\infty}{\longrightarrow}0,
    \]
    so $f_k\to f$ in $L^1_\textrm{loc}(\Omega)$. In particular, for $\varphi\in\Lipc(\Omega;\R^n)$ with $|\varphi|\leq w$,
    \[
        \int_\Omega f\div\varphi\,dx=\lim_{k\to\infty}\int_\Omega f_k\div\varphi\,dx.
    \]
    The remainder of the proof follows analogously to \cite[Theorem 5.2]{EG}.
\end{proof}

With this result in hand, we quickly remark that $BV(\Omega;w)$ is Banach.

\begin{lemma}
    Let $w:\R^n\to(0,\infty]$ be lower semicontinuous. $BV(\Omega;w)$ is a Banach space under the norm
    \begin{equation}
    \label{eqn:BVnorm}
        \lVert f\rVert_{BV(\Omega;w)}=\lVert f\rVert_{L^1(\Omega;w)}+\Dfw(\Omega).
    \end{equation}
\end{lemma}

\begin{proof}
    It is easy to see that \eqref{eqn:BVnorm} is a norm. Thus, it remains to show completeness. To that end, suppose $\{f_k\}_{k=1}^\infty\subseteq BV(\Omega;w)$ is Cauchy. Let $\e>0$. Then, there exists some $K\in\N$ such that for all $j,k>K$,
    \[
        \lVert f_j-f_k\rVert_{L^1(\Omega;w)}+\lVert D(f_j-f_k)\rVert_w(\Omega)=\lVert f_j-f_k\rVert_{BV(\Omega;w)}<\e.
    \]
    Hence, $\{f_k\}_{k=1}^\infty$ is Cauchy in $L^1(\Omega;w)$. Thus, there exists some $f\in L^1(\Omega;w)$ such that $f_k\to f$ in $L^1(\Omega;w)$.  Now, by Theorem \ref{thm:EG5.2analogue}, for $k>K$,
    \[
        \lVert D(f-f_k)\rVert_w(\Omega)\leq\liminf_{j\to\infty}\lVert D(f_j-f_k)\rVert_w(\Omega)<\e.
    \]
    Thus, $\lVert D(f-f_k)\rVert_w(\Omega)\to0$ as $k\to\infty$. Combined with the fact that $f_k\to f$ in $L^1(\Omega;w)$, this gives us that
    \[
        \lVert f-f_k\rVert_{BV(\Omega;w)}=\lVert f-f_k\rVert_{L^1(\Omega;w)}+\lVert D(f-f_k)\rVert_w(\Omega)\underset{k\to\infty}{\longrightarrow}0,
    \]
    so $f_k\to f$ in $BV(\Omega;w)$. Thus, $BV(\Omega;w)$ is complete.
\end{proof}

Now, we turn our attention to proving Theorem \ref{thm:EG5.3analogue}, that is, approximating functions in $BV(\Omega;w)$ by smooth functions.

\begin{definition}
\label{def:approximable}
    Let $w\in A_1^*$, $f\in BV(\Omega;w)$. We say that $f$ is \textbf{$w$-approximable} if
    \begin{equation}
    \label{eqn:wLebesguepts}
         \lim_{\e\to0}\fint_{B(x,\e)}|w(y)-w(x)|\,dy=0\qquad\textrm{for }\Df\textrm{-a.e. }x.
    \end{equation}
\end{definition}

A few remarks are in order to explain the $w$-approximability condition.

\begin{remark}
\label{rmk:wapprox}
    Note that condition (\ref{eqn:wLebesguepts}) is quite a general condition. It simply says that $\Df$-a.e. point is a Lebesgue point of $w$. Intuitively, it ensures that $w$ behaves nicely on the support of the part of $\Df$ that is mutually singular with the Lebesgue measure. For example, if $f\in W^{1,1}_{\text{loc}}(\Omega,w)$, then $d\Df=|Df|\,dx$. In this case, (\ref{eqn:wLebesguepts}) is satisfied by the Lebesgue Differentiation Theorem. Moreover, if \textit{every} point in $\Omega$ is a Lebesgue point of $w$ (e.g. if $w$ is continuous or a power weight), then (\ref{eqn:wLebesguepts}) holds for every $f\in BV(\Omega;w)$.
\end{remark}

\begin{remark}
    We remark here that the condition that $f$ is $w$-approximable is sufficient but not necessary to obtain the convergence $\Dfkw(\Omega)\to\Dfw(\Omega)$. For example, consider the $A_1^*$ weight
    \[
        w(x)=\begin{cases}
            1&\textrm{if }x\leq0\\
            2&\textrm{if }x>0,
        \end{cases}
    \]
    the $BV(\R;w)$ function $f=\chi_{(0,1)}$, and the smooth functions $f_k=\eta_{1/k}*\chi_{(-1/k,1)}$, where $\eta$ is the standard mollifier. Note that
    \[
        \begin{cases}
            f_k-f=f_k\cdot\chi_{(0,1)^c}\\
            \textrm{spt}(f_k)\subseteq[-2/k,1+1/k]\\
            0\leq f_k\leq 1.
        \end{cases}
    \]
    Hence,
    \[
        \int_{\R} |f_k-f|\,w\,dx\leq\int_{\R}\chi_{[-2/k,0]\cup[1,1+1/k]}\cdot w\,dx=\frac{4}{k}\underset{k\to\infty}{\longrightarrow}0.
    \]
    Thus, $f_k\to f$ in $L^1(\R;w)$. Moreover, for all $k\in\N$,
    \begin{align*}
        \Dfkw(\R)&=\int_{\R}\left|\frac{d}{dx}\int_{\R}\eta_{1/k}(x-y)\chi_{(-1/k,1)}(y)\,dy\right|\,w(x)\,dx\\
        &=\int_{\R}\left|\int_{-1/k}^1\frac{d\eta_{1/k}}{dx}(x-y)\,dy\right|\,w(x)\,dx\\
        &=\int_{\R}\left|\int_{-1/k}^1\frac{d\eta_{1/k}}{dy}(x-y)\,dy\right|\,w(x)\,dx\\
        &=\int_{\R}\left|\eta_{1/k}(x-1)-\eta_{1/k}(x+1/k)\right|\,w(x)\,dx\\
        &=3,
    \end{align*}
    and $\Dfw(\R)=3$, so certainly $\Dfkw(\R)\to\Dfw(\R)$. However, $\Df(\{0\})=1>0$ and
    \[
        \lim_{\e\to0}\fint_{B(0,\e)}|w(y)-w(0)|\,dy=\frac{1}{2}\neq0,
    \]
    so $f$ is not $w$-approximable.
\end{remark}

\begin{remark}
    Although the condition that $f$ is $w$-approximable is not necessary, the conclusion of Theorem \ref{thm:EG5.3analogue}(i) is not true for general $f$ and $w$. Indeed, consider the $A_1^*$ weight
    \[
        w(x)=\begin{cases}
            1&\textrm{if }x=0\textrm{ or }x=1\\
            2&\textrm{otherwise},
        \end{cases}
    \]
    and the $BV(\R;w)$ function $f=\chi_{(0,1)}$. For the sake of obtaining a contradiction, suppose $\{f_k\}_{k=1}^\infty\subseteq C^\infty(\R)\cap BV(\R;w)$ such that $f_k\to f$ in $L^1(\R;w)$ and $\Dfkw(\R)\to\Dfw(\R)$. Then,
    \[
        2\Dfk(\R)=2\int_\R |Df_k|\,dx=\int_{\R}|Df_k|\,w\,dx=\Dfkw(\R)\to\Dfw(\R)=2,
    \]
    so
    \begin{equation}
    \label{eqn:Dfkto1}
        \Dfk(\R)\to 1.
    \end{equation}
    Note that since $f_k\to f$ in $L^1(\R;w)$ and $w\approx 1$, we actually have that $f_k\to f$ in $L^1(\R)$, so there exists a subsequence $\{f_{k_j}\}_{j=1}^\infty$ such that $f_{k_j}\to f$ pointwise a.e. on $\R$. Then, there exists some $x_1\in(-\infty,0)$, $x_2\in(0,1)$, and $x_3\in (1,\infty)$ such that $f_{k_j}(x_1)\to f(x_1)=0$, $f_{k_j}(x_2)\to f(x_2)=1$ and $f_{k_j}(x_3)\to f(x_3)=0$. Then, using the definition of variation for real-valued functions on $\R$ (see \cite[Definition 5.11]{EG}),
    \[
        \lVert Df_{k_j}\rVert(\R)\geq\lVert Df_{k_j}\rVert([x_1,x_3])\geq|f_{k_j}(x_3)-f_{k_j}(x_2)|+|f_{k_j}(x_2)-f_{k_j}(x_1)|\underset{j\to\infty}{\longrightarrow}2,
    \]
    which contradicts \eqref{eqn:Dfkto1}. Thus, the conclusion of Theorem \ref{thm:EG5.3analogue}(i) is not true for any smooth approximation for this choice of $f$ and $w$.
\end{remark}

\begin{remark}
    Although the $w$-approximability condition is not optimal to obtain the conclusion of Theorem \ref{thm:EG5.3analogue}(i), it is natural since it will allow us to use mollification as our method of proof.
\end{remark}

For the proof of Theorem \ref{thm:EG5.3analogue}(i), we fix the following notation:
\[
    \Omega_\e:=\{x\in\Omega:\textrm{dist}(x,\partial\Omega)>\e\},\qquad\textrm{and}\qquad I_\e(E)=\{x:\textrm{dist}(x,E)<\e\}.
\]
To prove Theorem \ref{thm:EG5.3analogue}(i), we will also make use of the following result from \cite{AFP}.

\begin{lemma}[{\cite[Proposition 3.2]{AFP}}]
\label{lem:AFPprop3.2}
    Suppose $f\in\BVloc(\Omega)$. Then,
    \begin{enumerate}[(i)]
        \item for all $\psi\in\Lipc(\Omega)$, $f\psi\in\BVloc(\Omega)$, and $[D(f\psi)]=\psi\,[Df]+f\,D\psi\,dx$, and
        \item $D(f*\eta_\e)=[Df]*\eta_\e$ in $\Omega_\e$,
    \end{enumerate}
    where $\eta_\e$ is the standard mollifier.
\end{lemma}

\begin{proof}[Proof of Theorem \ref{thm:EG5.3analogue}(i)]
    First, note that $\Omega$ can be written as the union of a countable family of bounded open sets $\Omega_k\subseteq\Omega$, $k\in\N$, such that each $\Omega_k$ has positive distance from the boundary of $\Omega$ and each point in $\Omega$ belongs to at most $4$ sets $\Omega_k$. This follows from a standard construction that can be found in \cite{AFP} or \cite{EG}. We next choose a partition of unity with respect to the covering $\Omega_k$, that is, positive functions $\zeta_k\in C_c^\infty(\Omega_k)$ such that $\sum_k\zeta_k\equiv1$ on $\Omega$. Fix $\e>0$ and notice that for each $k\geq 1$ there exists $\e_k>0$ such that
    \[
        \begin{cases}
            \e_k<\e,\\
            \text{spt}((f\zeta_k)\ast\eta_{\e_k})\subseteq\Omega_k,\\
            I_{\e_k}(\Omega_k)\subseteq\Omega,\\
            \int_\Omega |(f\zeta_k)\ast\eta_{\e_k}-f\zeta_k|\,w\,dx<2^{-k}\e,\\
            \int_\Omega|(fD\zeta_k)\ast\eta_{\e_k}-fD\zeta_k|\,w\,dx<2^{-k}\e.
        \end{cases}
    \]
    The last two conditions follow from a standard fact about approximate identities in $L^1(\Omega;w)$ for $A_1$ weights $w$. Now, define
    \[
        f_\e:=\sum_{k=1}^\infty(f\zeta_k)*\eta_{\e_k}\in C^\infty(\Omega).
    \]
    Note also that
    \[
        f:=\sum_{k=1}^\infty f\zeta_k.
    \]
    Then, we have that
    \[
        \int_\Omega |f_\e-f|\,w\,dx\leq\sum_{k=1}^\infty\int_\Omega|(f\zeta_k)*\eta_{\e_k}-f\zeta_k|\,w\,dx<\e,
    \]
    so $f_\e\to f$ in $L^1(\Omega;w)$ as $\e\to0$.

    Now, by Lemma \ref{lem:AFPprop3.2} and using the facts that $I_{\e_k}(\Omega_k)\subseteq\Omega$ and $\sum_{k=1}^\infty D\zeta_k\equiv0$, we obtain that
    \begin{align*}
        Df_\e &=\sum_{k=1}^\infty D((f\zeta_k)*\eta_{\e_k})\\
        &=\sum_{k=1}^\infty [D(f\zeta_k)]*\eta_{\e_k}\\
        &=\sum_{k=1}^\infty (\zeta_k[Df])*\eta_{\e_k}+\sum_{k=1}^\infty (fD\zeta_k)*\eta_{\e_k}\\
        &=\sum_{k=1}^\infty (\zeta_k[Df])*\eta_{\e_k}+\sum_{k=1}^\infty((fD\zeta_k)*\eta_{\e_k}-fD\zeta_k)
    \end{align*}
    in $\Omega$. Then, we obtain
    \begin{align*}
        &\Dfew(\Omega)-\Dfw(\Omega)\\
        &\qquad\qquad=\int_\Omega |Df_\e|\,w\,dx-\Dfw(\Omega)\\
        &\qquad\qquad\leq\sum_{k=1}^\infty\int_\Omega|(\zeta_k[Df])*\eta_{\e_k}|w\,dx+\e-\Dfw(\Omega)\\
        &\qquad\qquad=\sum_{k=1}^\infty\int_\Omega\left|\int_\Omega\eta_{\e_k}(x-y)\zeta_k(y)\,d[Df](y)\right|\,w(x)\,dx+\e-\lVert Df\rVert_w(\Omega)\\
        &\qquad\qquad\leq\sum_{k=1}^\infty\int_{I_{\e_k}(\Omega_k)}\int_{\Omega_k}\eta_{\e_k}(x-y)\zeta_k(y)\,d\Df(y)\,w(x)\,dx+\e-\Dfw(\Omega)\\
        &\qquad\qquad=\sum_{k=1}^\infty\int_{\Omega_k}\int_{I_{\e_k}(\Omega_k)}\eta_{\e_k}(x-y)\zeta_k(y)\,w(x)\,dx\,d\Df(y)+\e-\Dfw(\Omega)\\
        &\qquad\qquad\leq\sum_{k=1}^\infty\int_{\Omega_k}(\eta_{\e_k}*w)\zeta_k\,d\Df+\e-\sum_{k=1}^\infty\int_{\Omega_k}w\zeta_k\,d\Df\\
        &\qquad\qquad=\sum_{k=1}^\infty\int_{\Omega_k}(\eta_{\e_k}*w-w)\zeta_k\,d\Df+\e.
    \end{align*}
    Now, since $w\in A_1$, Lemma \ref{lem:mollificationbound} implies that
    \[
        |\eta_{\e_k}*w-w|\leq ([w]_{A_1}+1)w,
    \]
    and so for all $k\in\N$ and $\e>0$,
    \[
        \int_{\Omega_k}([w]_{A_1}+1)w\,d\Df\leq([w]_{A_1}+1)\Dfw(\Omega_k)\leq([w]_{A_1}+1)\Dfw(\Omega)<\infty.
    \]
    Moreover, since each point in $\Omega$ belongs to at most four of the $\Omega_k$, we have
    \[
        \sum_{k=1}^\infty\left|([w]_{A_1}+1)\Dfw(\Omega_k)\right|\leq 4([w]_{A_1}+1)\Dfw(\Omega)<\infty,
    \]
    Thus, applying the Dominated Convergence Theorem twice yields that
    \[
        \limsup_{\e\to0}\left(\sum_{k=1}^\infty\int_{\Omega_k}(\eta_{\e_k}*w-w)\zeta_k\,d\Df+\e\right)=\sum_{k=1}^\infty\int_{\Omega_k}\limsup_{\e\to0}(\eta_{\e_k}*w-w)\zeta_k\,d\Df.
    \]
    Thus,
    \begin{align*}
        &\limsup_{\e\to0}\Dfew(\Omega)-\Dfw(\Omega)\\
        &\qquad\qquad\leq\sum_{k=1}^\infty\int_{\Omega_k}\limsup_{\e\to0}(\eta_{\e_k}*w-w)\zeta_k\,d\Df\\
        &\qquad\qquad=\sum_{k=1}^\infty\int_{\Omega_k}\limsup_{\e\to0}\left(\int_{B(x,\e_k)} \eta_{\e_k}(x-y)w(y)\,dy-w(x)\right)\zeta_k(x)\,d\Df(x)\\
        &\qquad\qquad=\sum_{k=1}^\infty\int_{\Omega_k}\limsup_{\e\to0}\left(\int_{B(x,\e_k)} \eta_{\e_k}(x-y)(w(y)-w(x))\,dy\right)\zeta_k(x)\,d\Df(x)\\
        &\qquad\qquad\lesssim\sum_{k=1}^\infty\int_{\Omega_k}\limsup_{\e\to0}\left(\fint_{B(x,\e_k)}|w(y)-w(x)|\,dy\right)\zeta_k(x)\,d\Df(x)\\
        &\qquad\qquad=0,
    \end{align*}
    where in the last equality we used the approximability condition (\ref{eqn:wLebesguepts}) and the fact that $\e_k\to 0$ as $\e\to 0$.
    
    On the other hand, it follows from Theorem \ref{thm:EG5.2analogue} that
    \[
        \Dfw(\Omega)\leq\liminf_{\e\to0}\Dfew(\Omega).
    \]
    This completes the proof.
\end{proof}

\begin{proof}[Proof of Theorem \ref{thm:EG5.3analogue}(ii)]
    The proof of Theorem \ref{thm:EG5.3analogue}(ii) works almost verbatim from the proof of \cite[Theorem 5.3]{EG} with only a few small modifications, which we will make note of here.

    First, we modify \cite[Equation $(\star\star)$, p. 200]{EG} to instead choose $\e_k>0$ for each $k\in\N$ such that
    \begin{equation}
    \label{eqn:EG5.3star2}
        \begin{cases}
            \textrm{spt}(\eta_{\e_k}*(f\zeta_k))\subseteq V_k\\
            \int_\Omega |\eta_{\e_k}*(f\zeta_k)-f\zeta_k|\,w\,dx<\frac{\e}{2^k}\\
            \int_\Omega |\eta_{\e_k}*(fD\zeta_k)-fD\zeta_k|\,w\,dx<\frac{\e}{2^k}.
        \end{cases}
    \end{equation}
    Then, one can show that
    \[
        \Dfw(\Omega)\leq\liminf_{\e\to0}\Dfew(\Omega)
    \]
    analogously to the method in \cite{EG}.

    Moreover, for any $\varphi\in\Lipc(\Omega;\R^n)$ with $|\varphi|\leq w$, we can perform a computation that follows \cite{EG} verbatim to see that
    \begin{align*}
        \int_\Omega f_\e\div\varphi\,dx&=\int_\Omega f\div(\zeta_1(\eta_{\e_1}*\varphi))\,dx+\sum_{k=2}^\infty\int_\Omega f\div(\zeta_k(\eta_{\e_k}*\varphi))\,dx\\
        &\qquad\qquad-\sum_{k=1}^\infty\int_\Omega\varphi\cdot(\eta_{\e_k}*(fD\zeta_k)-fD\zeta_k)\,dx=:\textrm{I}_\e+\textrm{I\!I}_\e+\textrm{I\!I\!I}_\e.
    \end{align*}
    Note that by Lemma \ref{lem:mollificationbound},
    \begin{align*}
        \eta_{\e_k}*\varphi(x)&\leq\eta_{\e_k}*w(x)\leq [w]_{A_1}w(x).
    \end{align*}
    Hence, for all $k\in\N$,
    \[
        |\zeta_k(\eta_{\e_k}*\varphi)|\leq [w]_{A_1}w.
    \]
    Thus,
    \[
        |\textrm{I}_\e|=\left|\int_\Omega f\div(\zeta_1(\eta_{\e_1}*\varphi))\,dx\right|\leq [w]_{A_1}\Dfw(\Omega).
    \]
    Also, note that each point in $\Omega$ belongs to at most three of the sets $\{V_k\}_{k=1}^\infty$. Thus,
    \begin{align*}
        |\textrm{I\!I}_\e|&\leq\sum_{k=2}^\infty\left|\int_\Omega f\div(\zeta_k(\eta_{\e_k}*\varphi))\,dx\right|\leq\sum_{k=2}^\infty [w]_{A_1}\Dfw(V_k)\leq3[w]_{A_1}\Dfw(\Omega\setminus\Omega_1)<3[w]_{A_1}\e.
    \end{align*}
    For the third term, (\ref{eqn:EG5.3star2}) implies that
    \[
        |\textrm{I\!I\!I}_\e|\leq\sum_{k=1}^\infty\int_\Omega |\eta_{\e_k}*(fD\zeta_k)-fD\zeta_k|\,w\,dx<\e.
    \]
    Hence,
    \[
        \Dfew(\Omega)\leq [w]_{A_1}\Dfw(\Omega)+3[w]_{A_1}\e+\e<\infty,
    \]
    so $f_\e\in BV(\Omega;w)$. Moreover,
    \[
        \limsup_{\e\to0}\Dfew(\Omega)\leq [w]_{A_1}\Dfw(\Omega).
    \]
    Thus, up to a subsequence, we have that
    \[
        \Dfw(\Omega)\leq\lim_{\e\to0}\Dfew(\Omega)\leq [w]_{A_1}\Dfw(\Omega).
    \]
\end{proof}

%%%%%%%%%%%%%%%%%%%%%%%%%%%%%%%%%%%%%%%%%%%%%%%%
%%%%%%%%%%%%%%%%%%%%%%%%%%%%%%%%%%%%%%%%%%%%%%%%
%%%%%%%%%%%%%%%%%%%%%%%%%%%%%%%%%%%%%%%%%%%%%%%%
%%%%%%%%%%%%%%%%%%%%%%%%%%%%%%%%%%%%%%%%%%%%%%%%
%%%%%%%%%%%%%%%%%%%%%%%%%%%%%%%%%%%%%%%%%%%%%%%%
%%%%%%%%%%%%%%%%%%%%%%%%%%%%%%%%%%%%%%%%%%%%%%%%

%\newpage

\section{Weighted Isoperimetric Inequalities}
\label{sec:isoperimetricinequality}

In this section, we prove Theorem \ref{thm:GNSforweightedBV} and Corollary \ref{cor:globalweightedisoperimetric}. To do this, we make use of the following result due to P\'erez and Rela \cite{PR19}.

\begin{theorem}[Gagliardo-Nirenberg-Sobolev Inequality for $W^{1,1}(\R^n;\mu)$]
\label{thm:GNSforweightedsobolev}
    Let $\mu$ be a locally finite Borel measure for which $M\mu<\infty$ a.e.\footnote{A characterization of such measures $\mu$ can be found in Appendix \ref{sec:maximalfunctionappendix}.} Then, there exists a constant $C_1>0$ such that for all $f\in W^{1,1}(\R^n;\mu)$,
    \[
        \lVert f\rVert_{L^{1^*}(\R^n;\mu)}\leq C_1\lVert Df\rVert_{L^1(\R^n;(M\mu)^{1/1^*})},
    \]
    where $1^*=n/(n-1)$.
\end{theorem}
In particular, note that $d\mu=w\,dx$, where $w\in A_1$, satisfies the hypotheses of Theorem \ref{thm:GNSforweightedsobolev}. Because of the exponents in this inequality, the following lemmas will also be relevant.

\begin{lemma}
\label{lem:deltaapproximable}
    Let $w\in A_1^*$ and $f\in BV(\Omega;w)$. If $f$ is $w$-approximable, then $f$ is $w^\delta$-approximable for all $0<\delta<1$.
\end{lemma}

\begin{proof}
    Let $0<\delta<1$, and suppose $f$ is $w$-approximable. Fix $x\in\Omega$ so that the $w$-approximability condition (\ref{eqn:wLebesguepts}) holds. Note that, in particular, this implies that $0<w(x)<\infty$. Then, note that
    \[
        \left|w^\delta(y)-w^\delta(x)\right|=w^\delta(x)\left|\left(\frac{w(y)}{w(x)}\right)^\delta-1\right|\leq w^\delta(x)\left|\frac{w(y)}{w(x)}-1\right|=\frac{w^\delta(x)}{w(x)}|w(y)-w(x)|.
    \]
    Thus, for $\Df$-a.e. $x$,
    \begin{align*}
        \lim_{\e\to0}\fint_{B(x,\e)}|w^\delta(y)-w^\delta(x)|\,dy&\leq\frac{w^\delta(x)}{w(x)}\lim_{\e\to0}\fint_{B(x,\e)}|w(y)-w(x)|\,dy=0.
    \end{align*}
\end{proof}

\begin{lemma}
\label{lem:deltaBVcontainment}
    Let $w:\R^n\to(0,\infty]$ be lower semicontinuous, and $0<\delta<1$. Then, $BV(\Omega;w)\subseteq\BVloc(\Omega;w^\delta)$.
\end{lemma}

\begin{proof}
    Let $f\in BV(\Omega;w)$, $V\Subset\Omega$, and set
    \[
        c_V:=\inf_{x\in V}w(x)>0.
    \]
    Then,
    \[
        w^\delta=c_V^\delta\left(\frac{w}{c_V}\right)^\delta\leq c^\delta_V\frac{w}{c_V}=c_V^{\delta-1}w,
    \]
    where we used the fact that $w/c_V\geq1$. Thus,
    \[
        \int_V w^\delta\,d\Df\leq c_V^{\delta-1}\int_V w\,d\Df<\infty,
    \]
    where we used the fact that $w\in L^1(\Omega;d\Df)$ from Theorem \ref{thm:varwmeasure}(i). Since $V\Subset\Omega$ was arbitrary, this implies that $w^\delta\in L^1_\text{loc}(\Omega;d\Df).$ With this fact in hand, and noting that $f\in BV(\Omega;w)\subseteq\BVloc(\Omega)$ by Lemma \ref{lem:BVwcontainment}(i), Theorem \ref{thm:varwmeasure}(ii) implies that $f\in\BVloc(\Omega;w^\delta)$. This shows the desired containment.
\end{proof}

\begin{lemma}[Minor Modification of Theorem \ref{thm:EG5.3analogue}]
\label{lem:EG5.3deltaanalogue}
    Let $w\in A_1^*$, $f\in BV(\Omega;w)$, and $0<\delta<1$.
    \begin{enumerate}[(i)]
        \item If $f$ is $w^\delta$-approximable, then there exists a sequence $\{f_k\}_{k=1}^\infty\subseteq\BVloc(\Omega;w^\delta)\cap C^\infty(\Omega)$ such that $f_k\to f$ in $L^1(\Omega;w)$ and
        \begin{equation}
        \label{eqn:wdeltalemma1}
            \limsup_{k\to\infty}\Dfk_{w^\delta}(\Omega)\leq\Df_{w^\delta}(\Omega).
        \end{equation}
        \item If $f$ is not $w^\delta$-approximable, then there exists a sequence $\{f_k\}_{k=1}^\infty\subseteq\BVloc(\Omega;w^\delta)\cap C^\infty(\Omega)$ such that $f_k\to f$ in $L^1(\Omega;w)$ and
        \begin{equation}
        \label{eqn:wdeltalemma2}
            \limsup_{k\to\infty}\Dfk_{w^\delta}(\Omega)\leq [w]_{A_1}^\delta\Df_{w^\delta}(\Omega).
        \end{equation}
    \end{enumerate}
\end{lemma}

\begin{proof}
    From Lemma \ref{lem:deltaBVcontainment}, we have that $f\in\BVloc(\Omega;w^\delta)$ and $f\in L^1_\textrm{loc}(\Omega;w^\delta)$. Now, we split into two cases.

    First, consider the case when $\Df_{w^\delta}(\Omega)=\infty$. If this happens, then we may choose the exact same sequence as in Theorem \ref{thm:EG5.3analogue}(i) or Theorem \ref{thm:EG5.3analogue}(ii), respectively, since the inequality (\ref{eqn:wdeltalemma1}) or (\ref{eqn:wdeltalemma2}), respectively, trivially holds.

    Otherwise, we assume that $\Df_{w^\delta}(\Omega)<\infty$. Then, we copy the proof of Theorem \ref{thm:EG5.3analogue}(i) or Theorem \ref{thm:EG5.3analogue}(ii), respectively, with the following modification. Namely, when we choose $\e_k$, we specify that
    \[
        \int_\Omega|\eta_{\e_k}*(fD\zeta_k)-fD\zeta_k)|\,w^\delta\,dx<\frac{\e}{2^k}.
    \]
    This is justified because we have that $f\in BV(\Omega;w)\subseteq BV_\textrm{loc}(\Omega;w^\delta)\subseteq L^1_\textrm{loc}(\Omega;w^\delta)$ by Lemma \ref{lem:deltaBVcontainment} and $D\zeta_k\in C_c^\infty(\Omega)$, so $f D\zeta_k\in L^1(\Omega;w^\delta)$, so the convolution converges in $L^1(\Omega;w^\delta)$. Then, we continue following the argument from Theorem \ref{thm:EG5.3analogue}, replacing $w$ by $w^\delta$ when necessary, to complete the proof.

        Note here that we use the fact that $[w^\delta]_{A_1}\leq[w]_{A_1}^\delta$. Indeed,
        \[
            \fint_B w^\delta\,dx\leq\left(\fint_B w\,dx\right)^\delta\leq\left([w]_{A_1}\inf_{x\in B}w(x)\right)^\delta=[w]_{A_1}^\delta\inf_{x\in B}w^\delta(x).
        \]
     %   \joseph{Is it true that $[w^\delta]_{A_1}=[w]_{A_1}^\delta$?}
    
\end{proof}

With these facts in hand, we can prove Theorem \ref{thm:GNSforweightedBV}, a Gagliardo-Nirenberg-Sobolev inequality for $BV(\R^n;w)$.

\begin{proof}[Proof of Theorem \ref{thm:GNSforweightedBV}]
    Choose a sequence of functions $\{f_k\}_{k=1}^\infty\subseteq C_c^\infty(\R^n)$ such that
    \[
        f_k\to f\textrm{ in }L^1(\Omega;w),\quad f_k\to f\,\,\mathcal{L}^n\textrm{-a.e.},\quad\limsup_{k\to\infty}\Dfk_{w^{1/1^*}}\leq [w]_{A_1}^{1/1*}\Df_{w^{1/1^*}}.
    \]
    Such functions exist according to Lemma \ref{lem:EG5.3deltaanalogue}. The compact support can be obtained by multiplying by smooth cutoff functions with ascending supports. The pointwise a.e. convergence can be assured by taking a subsequence if necessary.

    Now, Fatou's Lemma and Theorem \ref{thm:GNSforweightedsobolev} imply that
    \begin{align*}
        \lVert f\rVert_{L^{1^*}(\R^n;w)}&\leq\liminf_{k\to\infty}\lVert f_k\rVert_{L^{1^*}(\R^n;w)}\\
        &\leq C_1\limsup_{k\to\infty}\Dfk_{L^1(\R^n;(Mw)^{1/1^*})}\\
        &\leq C_1[w]_{A_1}^{1/1^*}\limsup_{k\to\infty}\Dfk_{L^1(\R^n;w^{1/1^*})}\\
        &\leq C_1[w]_{A_1}^{2/1^*}\Df_{w^{1/1^*}}(\R^n).
    \end{align*}

    If, in addition, $f$ is $w^{1/1^*}$-approximable, then according to Lemma \ref{lem:EG5.3deltaanalogue}, we may assume that
    \[
        \limsup_{k\to\infty}\Dfk_{w^{1/1^*}}\leq\Df_{w^{1/1^*}}.
    \]
    Then, the chain of inequalities becomes
    \begin{align*}
        \lVert f\rVert_{L^{1^*}(\R^n;w)}&\leq C_1[w]_{A_1}^{1/1^*}\limsup_{k\to\infty}\Dfk_{L^1(\R^n;w^{1/1^*})}\\
        &\leq C_1[w]_{A_1}^{1/1^*}\Df_{w^{1/1^*}}(\R^n).
    \end{align*}
    This completes the proof.
\end{proof}

\section{Isometrically Embedding $BV(\Omega;w)\hookrightarrow BV(\Omega_w)$}
\label{sec:embedding}

In this section, we prove Theorem \ref{thm:isometry}. To begin, we state a key definition.

\begin{definition}
\label{def:J}
    Let $\Omega\subseteq\R^n$ be an open set and $w:\mathbb{R}^n\to(0,\infty]$ be lower-semicontinuous. The \textbf{subgraph} of $w$ in $\Omega$ is given by
    \[
        \Omega_w=\{(x,y)\in\R^n\times\R:x\in\Omega,0<y<w(x)\}.
    \]
    It follows by the lower-semicontinuity of $w$ that the subgraph $\Omega_w$ is open. For $f\in L^1(\Omega;w)$, we define $Jf:\Omega_w\to\R$ by $Jf(x,y)=f(x)$.
\end{definition}

\begin{remark}
\label{rmk:AntociRemark}
    Following \cite[Section 4]{A}, we have that $J:W^{1,1}(\Omega;w)\to W^{1,1}(\Omega_w)$ is a well-defined isometric embedding. That is,
    \[
        \lVert f\rVert_{L^1(\Omega;w)}=\lVert Jf\rVert _{L^1(\Omega_w)}\qquad\textrm{and}\qquad\lVert Df\rVert_{L^1(\Omega;w)}=\lVert D(Jf)\rVert_{L^1(\Omega_w)}.
    \]
    More generally, $J:L^1(\Omega;w)\to L^1(\Omega_w)$ is a well-defined isometry.\footnote{To prove this, just use Fubini's Theorem.}
\end{remark}

We would like to extend this result to $BV(\Omega;w)$. Such a result could be a useful tool to turn problems in a weighted $BV$ space into problems in the unweighted embedding. To that end, we first present the following lemma for sets of finite $w$-perimeter.

\begin{lemma}
\label{lem:b3}
    Let $w:\mathbb{R}^n\to(0,\infty]$ be lower semicontinuous and let $\Omega\subseteq\mathbb{R}^n$ be open. If $E\subseteq\mathbb{R}^n$ has finite $w$-perimeter in $\Omega$, then $E_w=\{(x,y)\in\mathbb{R}^{n+1}:x\in E,0<y<w(x)\}$ has finite perimeter in $\Omega_w$ and 
    \[
        \lVert\partial E\rVert_w(\Omega)=\lVert\partial E_w\rVert(\Omega_w).
    \] 
\end{lemma}

\begin{remark}
    In the following proof, instead of denoting the $n$-dimensional Lebesgue measure of $E$ by $|E|$, we will denote it by $\mathcal{L}^n(E)$ to make the dimension of the ambient space obvious. Moreover, by $Q_r(x)$, we mean the cube in $\R^n$ centered at $x$ with side length $2r$, and by $Q_r(x,y)$, we mean the cube in $\R^n\times\R$ centered at $(x,y)$ with side length $2r$.
\end{remark}

\begin{proof}
    First, we claim that $(\partial_*E_w)\cap\Omega_w =\{(x,y)\in\mathbb{R}^{n+1}:x\in(\partial_*E)\cap \Omega,0<y<w(x)\}$. To that end, suppose $(x,y)\in (\partial_*E_w)\cap\Omega_w$. Then, $(x,y)\in \Omega_w$, so $0<y<w(x)$ and $x\in \Omega$. Recall that an equivalent definition for $(x,y)$ being in the measure theoretic boundary of $E_w$, namely $\partial_*E_w$, is that 
    \[
        \limsup_{r\to 0}\frac{\mathcal{L}^{n+1}(Q_r(x,y))\cap E_w)}{r^{n+1}}>0\qquad\textrm{and}\qquad\limsup_{r\to 0}\frac{\mathcal{L}^{n+1}(Q_r(x,y))\cap E_w^c)}{r^{n+1}}>0.
    \]
    Since $\Omega_w$ is open (see Definition \ref{def:J}), we have that for small enough $r$, $Q_r(x,y)\subseteq\Omega_w$. Therefore, for small enough $r$ and $(s,t)\in Q_r(x,y)$, we have that $(s,t)\in E_w$ if and only if $s\in E$. We now have that for small enough $r$, 
    \begin{align*}
        \frac{\mathcal{L}^{n+1}(Q_r(x,y)\cap E_w)}{r^{n+1}}&=\frac{1}{r^{n+1}}\int_{Q_r(x,y)}\chi_{E_w}(s,t)\,d(s,t)\\
        &=\frac{1}{r^{n+1}}\int_{Q_r(x)}\int_{y-r}^{y+r} \chi_E(s)\,dt\,ds\\
        &=\frac{2r}{r^{n+1}}\int_{Q_r(x)}\chi_E(s)\,ds\\
        &=\frac{2\mathcal{L}^n(Q_r(x)\cap E)}{r^n}.
    \end{align*}
    Similarly, for small enough $r$,
    \[
        \frac{\mathcal{L}^{n+1}(Q_r(x,y)\cap E_w^c)}{r^{n+1}}=\frac{2\mathcal{L}^n(Q_r(x)\cap E^c)}{r^n}
    \] 
    It follows that 
    \[
        \limsup_{r\to 0}\frac{\mathcal{L}^n(Q_r(x)\cap E)}{r^n}>0\qquad\textrm{and}\qquad\limsup_{r\to 0}\frac{\mathcal{L}^n(Q_r(x)\cap E^c)}{r^n}>0.
    \]
    Thus, $x\in\partial_* E$, so $(x,y)\in \{(x,y)\in\mathbb{R}^{n+1}:x\in\partial_*E,0<y<w(x)\}$. Thus, $(\partial_*E_w)\cap\Omega_w\subseteq\{(x,y)\in\mathbb{R}^{n+1}:x\in(\partial_*E)\cap \Omega,0<y<w(x)\}$. The reverse containment can be obtained analogously. This proves the claim.

    %%%%%%%%%%%%%%%%%%%%%%%%%%%%%%%%%%
    \iffalse
    Now suppose that $(x,y)\in\{(x,y)\in\mathbb{R}^{n+1}:x\in\partial_*E\cap \Omega,0<y<w(x)\}$. Since $x\in \Omega$ and $0<y<w(x)$, it follows that $(x,y)\in \Omega_w$. Since $x\in \partial_* E$, 
    \[
        \limsup_{r\to 0}\frac{\mathcal{L}^{n}(Q(x,r)\cap E)}{r^{n}}>0\qquad\textrm{and}\qquad\limsup_{r\to 0}\frac{\mathcal{L}^{n}(Q(x,r)\cap E)}{r^{n+1}}>0.
    \]
    For small enough $r$ and $(s,t)\in Q((x,y),r)$, we have that $(s,t)\in E_w$ if and only if $s\in E$. We then have that for small enough $r$, 
    \begin{align*}
        \frac{\mathcal{L}^{n}(Q(x,r)\cap E)}{r^{n}}&=\frac{2r}{2r^{n+1}}\int_{Q(x,r)}\chi_{E}(s)\,ds\\
        &=\frac{1}{2r^{n+1}}\int_{Q(x,r)}\int_{y-r}^{y+r} \chi_E(s)\,dtds\\
        &=\frac{1}{2r^{n+1}}\int_{Q(x,r)}\chi_{E_w}(s,t)\,dtds\\
        &=\frac{\mathcal{L}^{n+1}(Q((x,y),r)\cap E_w)}{2r^{n+1}}.
    \end{align*}
    Similarly, for small enough $r$,
    \[
        \frac{\mathcal{L}^{n}(Q(x,r)\cap E^c)}{r^{n}}=\frac{\mathcal{L}^{n+1}(Q((x,y),r)\cap E_w^c)}{2r^{n+1}}.
    \]
    It follows that 
    \[
        \limsup_{r\to 0}\frac{\mathcal{L}^{n+1}(Q((x,y),r)\cap E_w)}{r^{n+1}}>0\qquad\textrm{and}\qquad\limsup_{r\to 0}\frac{\mathcal{L}^{n+1}(Q((x,y),r)\cap E_w^c)}{r^{n+1}}>0.
    \]
    Therefore, $(x,y)\in\partial_*E_w$ such that $(x,y)\in (\partial_*E_w)\cap\Omega_w$. It follows that $(\partial_*E_w)\cap\Omega_w =\{(x,y)\in\mathbb{R}^{n+1}:x\in(\partial_*E)\cap \Omega,0<y<w(x)\}$.
    \fi
    %%%%%%%%%%%%%%%%%%%%%%%%%%%%%%%%%%%%%%%%%
    
    With this claim in hand, we will now obtain our result. Since $E$ has finite $w$-perimeter in $\Omega$, $E$ has locally finite perimeter in $\Omega$ by Lemma \ref{lem:BVwcontainment}. Moreover, by \cite[Theorem 5.16]{EG}, we know that $\lVert\partial E\rVert=\mathcal{H}^{n-1}\mres\partial_*E$. By these facts and Theorem \ref{thm:varwmeasure}, we have
    \[
        \lVert\partial E\rVert_w(\Omega)=\int_{(\partial_* E)\cap\Omega}w\,d\mathcal{H}^{n-1}.
    \]
    Since $w$ is lower semicontinuous, $w$ is measurable. By this and the fact that $w$ is positive, there exist an increasing sequence of functions $w_j=\sum_{k=1}^\infty a_{j,k}\chi_{F_{j,k}}$,  such that $w_j\to w$ and for all $j\in\mathbb{N}$, 
    \begin{equation}
    \label{eqn:wjbounds}
        \int_{(\partial_*E)\cap \Omega}w_j\,d\mathcal{H}^{n-1}\leq \int_{(\partial_*E)\cap \Omega} w\,d\mathcal{H}^{n-1}\leq \int_{(\partial_*E)\cap \Omega}w_j\,d\mathcal{H}^{n-1}+\frac{1}{j}.
    \end{equation}
    We can also assume that for each $j\in\mathbb{N}$, the constants $a_{j,k}$ are positive and the sets $F_{j,k}$ are disjoint and Borel. A short calculation shows that
    \[
        \int_{(\partial_*E)\cap \Omega}w_j\,d\mathcal{H}^{n-1}=\sum_{k=1}^\infty a_{j,k}\mathcal{H}^{n-1}((\partial_*E)\cap\Omega\cap F_{j,k}).
    \]
    Notice that, without loss of generality, we can assume that $E$ is Borel. Otherwise, there exists a Borel set $E'$ such that $\chi_E=\chi_{E'}$ $\mathcal{L}^n$-a.e. It follows that $\chi_{E_w}=\chi_{E'_w}$ $\mathcal{L}^{n+1}$-a.e. We trivially have that $\lVert \chi_E\rVert_{L^1(\Omega,w)}=\lVert \chi_{E'}\rVert_{L^1(\Omega,w)}$ and $\lVert \chi_{E_w}\rVert_{L^1(\Omega_w)}=\lVert \chi_{E'_w}\rVert_{L^1(\Omega_w)}$. By their definitions, both the weighted and unweighted variation measures are invariant under changes of the function on a null set. Therefore, $\lVert\partial E\rVert_w (\Omega)=\lVert\partial E'\rVert_w (\Omega)$ and $\lVert\partial E_w\rVert (\Omega_w)=\lVert\partial E'_w\rVert(\Omega_w)$. With this assumption in mind, it follows that $\partial_*E$ is Borel. By \cite[Theorem 5.15 and Lemma 5.5]{EG}, we know that $\partial_*E$ is countably $(n-1)$-rectifiable. It follows that that $(\partial_*E)\cap\Omega\cap F_{j,k}$ is countably $(n-1)$-rectifiable and Borel. Therefore, by \cite[Theorem 3.2.23]{Fed69}, we have that
    \[
        a_{j,k}\mathcal{H}^{n-1}((\partial_*E)\cap\Omega\cap F_{j,k})=\mathcal{H}^n\left(\{(x,y)\in\mathbb{R}^{n+1}:x\in(\partial_*E)\cap\Omega\cap F_{j,k},0<y<a_{j,k}\}\right).
    \]
    Since the sets $F_{j,k}$ are disjoint for each $j\in\mathbb{N}$, we have that 
    \begin{align}
    \label{eqn:wjequality}
        \int_{(\partial_*E)\cap \Omega}w_j\,d\mathcal{H}^{n-1}&=\sum_{k=1}^\infty\mathcal{H}^n\left(\{(x,y)\in\mathbb{R}^{n+1}:x\in(\partial_*E)\cap\Omega\cap F_{j,k},0<y<a_{j,k}\}\right)\nonumber\\
        &=\mathcal{H}^n\left(\bigcup_{k=1}^\infty\{(x,y)\in\mathbb{R}^{n+1}:x\in(\partial_*E)\cap\Omega\cap F_{j,k},0<y<a_{j,k}\}\right)\\
        &=\mathcal{H}^n\left(\{(x,y)\in\mathbb{R}^{n+1}:x\in(\partial_* E)\cap\Omega,0<y<w_j(x)\}\right).\nonumber
    \end{align}
    Since $w_j\nearrow w$, we have that
    \begin{align}
    \label{eqn:wjlimit}
        &\lim_{j\to\infty}\mathcal{H}^n\left(\{(x,y)\in\mathbb{R}^{n+1}:x\in(\partial_* E)\cap\Omega,0<y<w_j(x)\}\right)\nonumber\\
        &\qquad=\mathcal{H}^n\left(\bigcup_{j=1}^\infty\{(x,y)\in\mathbb{R}^{n+1}:x\in(\partial_* E)\cap\Omega,0<y<w_j(x)\}\right)\\
        &\qquad=\mathcal{H}^n\left(\{(x,y)\in\mathbb{R}^{n+1}:x\in(\partial_*E)\cap \Omega,0<y<w(x)\}\right).\nonumber
    \end{align}
    Taking $j\to\infty$ in \eqref{eqn:wjbounds}, and using \eqref{eqn:wjequality} and \eqref{eqn:wjlimit}, we obtain
    \[
        \lVert\partial E\rVert_w(\Omega)=\int_{(\partial_*E)\cap \Omega} w\,d\mathcal{H}^{n-1}=\mathcal{H}^n\left(\{(x,y)\in\mathbb{R}^{n+1}:x\in(\partial_*E)\cap \Omega,0<y<w(x)\}\right).
    \]
    Recall that $\{(x,y)\in\mathbb{R}^{n+1}:x\in(\partial_*E)\cap \Omega,0<y<w(x)\}=(\partial_*E_w)\cap \Omega_w$. Since $\lVert\partial E\rVert_w(\Omega)=\mathcal{H}^n((\partial_*E_w)\cap\Omega_w)<\infty$, we have by \cite[Theorem 1.1]{La20} that $E_w$ has finite perimeter in $\Omega_w$. By \cite[Theorem 5.16]{EG}, $\lVert\partial E_w\lVert(\Omega_w)=\mathcal{H}^n((\partial_*E_w)\cap \Omega_w)$. Therefore,
    \[
        \lVert\partial E\rVert_w(\Omega)=\lVert\partial E_w\rVert(\Omega_w).
    \]
    This completes the proof.
\end{proof}

In order to extend this result from sets of finite $w$-perimeter to all functions in $BV(\Omega;w)$, we will need the following version of a coarea formula, variations of which are well documented by Camfield in \cite{C}.

\begin{theorem}[Minor Modification of {\cite[Theorem 3.1.13]{C}}]
\label{thm:Coarea}
Let $w:\mathbb{R}^n\to (0,\infty]$, and let $\Omega\subseteq\mathbb{R}^n$ be open. If $f\in L^1_{\text{loc}}(\Omega,w)$, we define for $t\in\mathbb{R}$ the sets $E_t=\{x\in\Omega:f(x)>t\}$. Then
\[
    \Dfw(\Omega)=\int_{-\infty}^\infty \lVert\partial E_t\rVert_w(\Omega)\,dt.
\]
It particular, if $f\in BV(\Omega;w)$, then $E_t$ has finite $w$-perimeter for a.e. $t\in\mathbb{R}$.
\end{theorem}

With these results in hand, we can prove Theorem \ref{thm:isometry}.

\begin{proof}[Proof of Theorem \ref{thm:isometry}]
    Fix $f\in BV(\Omega;w)$. First, note that 
    \[
        \int_\Omega |f|\,w\,dx=\int_\Omega\int_0^{w(x)} |Jf|(x,y)\,dy\,dx=\int_{\Omega_w} |Jf|(x,y)\,d(x,y).
    \]
    Therefore, $\lVert f\rVert_{L^1(\Omega,w)}=\lVert Jf\rVert_{L^1(\Omega,w)}$. We define $E_t=\{x\in\Omega:f(x)>t\}$ and $E_{t,w}=\{(x,y)\in\mathbb{R}^{n+1}:x\in E_t,0<y<w(x)\}$. It follows that $E_{t,w}=\{(x,y)\in\Omega_w:J(x,y)>t\}$. Since $w$ is positive, Theorem \ref{thm:Coarea} implies that
    \[
        \Dfw(\Omega)=\int_{-\infty}^\infty \lVert\partial E_t\lVert_w(\Omega)\,dt
    \]
    and that $E_t$ has finite $w$-perimeter for a.e. $t\in\mathbb{R}$. Furthermore, by \cite[Theorem 5.9]{EG}, we have
    \[
         \lVert D(Jf)\rVert(\Omega_w)=\int_{-\infty}^\infty \lVert\partial E_{t,w}\lVert(\Omega_w)\,dt.
    \]
    It follows by Lemma \ref{lem:b3} that 
    \begin{align*}
        \Dfw(\Omega)&=\int_{-\infty}^\infty \lVert\partial E_t\lVert_w(\Omega)\,dt\\
        &=\int_{-\infty}^\infty \lVert\partial E_{t,w}\lVert(\Omega_w)\,dt\\
        &=\lVert D(Jf)\rVert(\Omega_w).
    \end{align*}
    Then $Jf\in BV(\Omega_w)$. Finally, since 
    \[
        \lVert f\rVert_{L^1(\Omega;w)}=\lVert Jf\rVert _{L^1(\Omega_w)}\qquad\textrm{and}\qquad\Dfw(\Omega)=\lVert D(Jf)\rVert(\Omega_w),
    \]
    we have that $\lVert f\rVert_{BV(\Omega;w)}=\lVert Jf\rVert_{BV(\Omega_w)}$.
\end{proof}

\appendix 

\section{Characterization of $\mathcal{M}_F$}
\label{sec:maximalfunctionappendix}

Define the class of locally finite Borel measures for which the Hardy–Littlewood maximal function is finite almost everywhere. Let $M_{\textsf{loc}}(\R^n)$ denote the set of positive locally finite Borel measures, and set
$$\mathcal M_F=\{\mu\in M_\textrm{loc}(\R^n) : M\mu<\infty \ a.e.\}.$$

A classical result of Coifman and Rochberg \cite{CR} states that if $\mu \in \mathcal M_F$ and $0 \leq \delta < 1$, then the weight $w = (M\mu)^\delta$ belongs to $A_1$. Conversely, given any $A_1$ weight, there exists $\mu \in \mathcal M_F$ and $0 < \delta < 1$ such that $w \approx (M\mu)^\delta$ a.e. In addition, the weight $(M\mu)^\delta$ is an $A_1^*$ weight; that is, it is defined everywhere and lower semicontinuous. Thus, understanding the class $\mathcal M_F$ is fundamental for the construction of $A_1$ weights.
The class of $f\in L^1_{\textrm{loc}}(\R^n)$ for which $Mf<\infty$ a.e., has been studied by Fiorenza and Krbec \cite{FK}. We provide a complete characterization for measures in $\mathcal M_F$, with proofs that differ in from theirs.

%{\color{red} The condition should be an Ahlfors-David regularity, namely, for almost every $x\in \R^n$
%$$\mu(B(x,R))\leq C(x)R^n, \qquad R>0.$$
%Where the constant $C$ depends on $x$.}

\begin{theorem}[Characterization of $\mathcal{M}_F$]
\label{thm:question 2.4}
    Let $\mu$ be a locally finite Borel measure. Then the following are equivalent:
    \begin{enumerate}[(1)]
        \item there exists $x_0\in\mathbb{R}^N$ such that $(M\mu)(x_0)<\infty$;
        \item there exists $x_0\in\mathbb{R}^N$ such that
        \[
            \limsup_{R\to\infty}\frac{\mu(B(x_0,R))}{|B(x_0,R)|}<\infty;
        \]
        \item there exists $K>0$ such that 
        \[
            \limsup_{R\to\infty}\frac{\mu(B(x,R))}{|B(x,R)|}=K
        \]
        for all $x\in\mathbb{R}^N$;
        \item $M\mu<\infty$ a.e.
    \end{enumerate}
\end{theorem}

\begin{proof}
    $(4)\implies(1)$ is trivial. And $(1)\implies(2)$ holds by choosing the same value for $x_0$ in both cases.
    
    $(2)\implies(3)$. Suppose $(2)$ holds such that there exists $x_0\in\mathbb{R}^N$ with
    \[
        \limsup_{R\to\infty}\frac{\mu(B(x_0,R))}{|B(x_0,R)|}<\infty.
    \]
    Let $y$ be any point in $\mathbb{R}^N\setminus\{x_0\}$. Let $d=|x_0-y|$. For any $R>0$, we have that $B(y,R)\subseteq B(x_0,R+d)$. Therefore,
    \[
        \frac{\mu(B(y,R))}{|B(y,R)|}\leq \frac{|B(x_0,R+d)|}{|B(x_0,R)|}\frac{\mu(B(x_0,R+d))}{|B(x_0,R+d)|}.
    \]
    Taking the $\limsup$ on both sides, we obtain
    \[
        \limsup_{R\to\infty}\frac{\mu(B(y,R))}{|B(y,R)|}\leq\limsup_{R\to\infty}\frac{\mu(B(x_0,R))}{|B(x_0,R)|}.
    \]
    The other direction holds by interchanging the roles of $x_0$ and $y$. Thus, $(3)$ holds.

    $(3)\implies(4)$. Suppose $(3)$ holds. Then, note that for all $n\in\N$, $\mu_n:=\mu\llcorner B(0,n)$ is a finite Borel measure. Hence, $M\mu_n<\infty$ a.e. Let $E_1\subseteq B(0,1)$ be a measure zero set such that $(M\mu_1)(x)<\infty$ for all $x\in B(0,1)\setminus E_1$. Then, inductively choose $E_{n+1}\subseteq B(0,n+1)$ to be a measure zero set such that $E_n\subseteq E_{n+1}$ and $(M\mu_{n+1})(x)<\infty$ for all $x\in B(0,n+1)\setminus E_{n+1}$. Set $E=\bigcup_{i=1}^\infty E_i$. Then, $E$ has measure zero. Now, let $x\in\R^N\setminus E$. Then, $x\in B(0,n)\setminus E_n$ for some $n\in\N$. Let $r_0>0$ such that $B(x,r_0)\subseteq B(0,n)$. Then, for all $R\leq r_0$,
    \[
        \frac{\mu(B(x,R))}{|B(x,R)|}\leq (M\mu_n)(x)<\infty.
    \]
    Further, by (3), there exists some $R_0$ such that
    \[
        \frac{\mu(B(x,R))}{|B(x,R)|}<2K
    \]
    for all $R\geq R_0$. Finally, for all $R\in(r_0,R_0)$,
    \[
        \frac{\mu(B(x,R))}{|B(x,R)|}\leq\frac{\mu(B(x,R_0))}{|B(x,r_0)|}<\infty.
    \]
    Thus, $(M\mu)(x)<\infty$. Since $x$ was an arbitrary point in $\R^N\setminus E$, this implies (4).
\end{proof}

%%%%%%%%%%%%%%%%%%%%%%%%%%%%%%%%%%%%%%%%%%%%%%%%
%%%%%%%%%%%%%%%%%%%%%%%%%%%%%%%%%%%%%%%%%%%%%%%%
%%%%%%%%%%%%%%%%%%%%%%%%%%%%%%%%%%%%%%%%%%%%%%%%
%%%%%%%%%%%%%%%%%%%%%%%%%%%%%%%%%%%%%%%%%%%%%%%%
%%%%%%%%%%%%%%%%%%%%%%%%%%%%%%%%%%%%%%%%%%%%%%%%
%%%%%%%%%%%%%%%%%%%%%%%%%%%%%%%%%%%%%%%%%%%%%%%%


\newpage

\begin{thebibliography}{AFP00}
\parskip=0.1cm

\bibitem[AFP00]{AFP} L. Ambrosio, N. Fusco, and D. Pallara. \textit{Functions of Bounded Variation and Free Discontinuity Problems}. Oxford University Press, 2000.

\bibitem[An03]{A} F. Antoci. Some necessary and some sufficient conditions for the compactness of the embedding of weighted Sobolev spaces. Ricerche di Matematica 52 (2003) 55-71.

\bibitem[Ba01]{B} A. Baldi. Weighted BV Functions. \textit{Houston Journal of Mathematics} 27 (2001), no. 3, 683-705.

\bibitem[Ca08]{C} C. S. Camfield. Comparison of BV Norms in Weighted Euclidean Spaces and Metric Measure Spaces. Doctoral dissertation, University of Cincinnati, 2008.

\bibitem[CR80]{CR} R. Coifman and R. Rochberg, \emph{Another characterization of $BMO$} Proc. AMS., 79 (1980) 249-254.

\bibitem[EG15]{EG} L. C. Evans and R. F. Gariepy. \textit{Measure Theory and Fine Properties of Functions}, revised edition. CRC Press, Boca Raton, FL, 2015.

\bibitem[Fe69]{Fed69} H. Federer, Geometric Measure Theory, Springer-Verlag, New York, 1969.

\bibitem[FK00]{FK} A. Fiorenza and M. Krbec, \emph{On the domain and range of the maximal operator}, Nagoya Math (2000) 43-61.


\bibitem[Gi84]{G84}
E.~Giusti, \emph{Minimal Surfaces and Functions of Bounded Variation}, 
Monographs in Mathematics, vol.~80, Birkh{\"a}user, Basel, 1984.

\bibitem[Gr14]{Graf} L. Grafakos, \emph{Classical Fourier Analysis}, 3rd Ed., Graduate Texts in Mathematics, vol. 249, Springer, 2014.

\bibitem[La20]{La20} Lahti, Panu, \emph{Federer’s characterization of sets of finite perimeter in metric spaces}, Analysis \& PDE 13.5 (2020): 1501-1519.

\bibitem[HV75]{HV75} A.I. Hudjaev and S.I. Vol'pert, \emph{Analysis in classes of discontinuous functions and equations of mathematical physics}, Nauka, 1975.

\bibitem[PR19]{PR19} C. P\'erez and E. Rela, \emph{Degenerate Poincar\'e-Sobolev inequalities}, Trans. Amer. Math. Soc., 372 (2019) 6087--6133.


\end{thebibliography}
\end{document}